\documentclass[12pt]{article}

\usepackage{epsf,epsfig,amsmath,amsfonts,amsthm}
\usepackage{graphicx}
\usepackage{color}
\usepackage{accents}
\usepackage{mathrsfs}
\usepackage{tikz}
\usepackage{hyperref}

\newtheorem{theorem}{Theorem}
\newtheorem{corollary}[theorem]{Corollary}

\newtheorem{lemma}[theorem]{Lemma}
\newtheorem{definition}{Definition}
\newtheorem{remark}{Remark}
\newtheorem{assumption}{Assumption}

\usepackage{graphics} 
  \usepackage{epsfig}
\usepackage{lipsum}
\usepackage{amsfonts}
\usepackage{amssymb}
\usepackage{graphicx}
\usepackage{epstopdf}
\usepackage{algorithm}
\usepackage{algorithmic}
\usepackage{graphicx}
\usepackage{enumitem}
\usepackage{tikz}
\usepackage{booktabs}
\usepackage{float}
\usetikzlibrary{positioning,arrows.meta,shadows,fit}

\usepackage{subfigure}
\ifpdf
  \DeclareGraphicsExtensions{.eps,.pdf,.png,.jpg}
\else
  \DeclareGraphicsExtensions{.eps}
\fi

\usepackage{enumitem}

\newcommand{\N}{\mathbb{N}}
\newcommand{\D}{\mathbb{D}}
\newcommand{\B}{\mathbb{B}}

\newcommand{\E}{\mathbb{E}}
\newcommand{\HH}{\mathbb{H}}

\def\bbr{{\mathbb{R}}}
\def\bbe{{\mathbb{E}}}
\def\x{{\bf x}}
\def\y{{\bf y}}
\def\z{{\bf z}}
\def\v{{\bf v}}
\def\u{{\bf u}}

\def\q{{\bf q}}
\def\a{{\bf a}}

\usepackage{bm}

\DeclareMathOperator{\dist}{\operatorname{dist}}

\begin{document}

\title{Differential Stochastic Variational Inequalities with Parametric Optimization}

\author{
{ Xiaojun Chen\footnote{Department of Applied Mathematics, The Hong Kong Polytechnic University, Kowloon, Hong Kong (\texttt{maxjchen@polyu.edu.hk}).},
\ Jian Guo\footnote{Department of Applied Mathematics, The Hong Kong Polytechnic University, Kowloon, Hong Kong (\texttt{jiguo@polyu.edu.hk}).}
\ and Guan Wang\footnote{Department of Applied Mathematics, The Hong Kong Polytechnic University, Kowloon, Hong Kong (\texttt{guan1.wang@polyu.edu.hk}).}
}}

\maketitle
\begin{abstract} \noindent
{This paper is focused on studying the differential} stochastic variational inequality with
parametric convex optimization (DSVI-O), {that is,} an ordinary differential equation whose right-hand
side involves a stochastic variational inequality and solutions of several dynamic and random parametric convex optimization problems.
Let the involved random variable possess the time-dependent distribution, and be measurable and integrable.
We first show that the DSVI-O has a weak solution coupled with measurable and integrable solutions of the parametric  optimization problems. Next we  propose a discrete scheme of DSVI-O by using the time-stepping approximation and sample average approximation
 and then prove its convergence. Finally, we highlight the superiority of our theoretical results on DSVI-O over the existing ones through applications to an embodied intelligence system for the elderly health, using synthetic health care data generated by Multimodal Large Language Models.
\\\\
\noindent
{\em MSC codes\/}: 90C15, 90C33, 90C39
\\\\
\noindent {\em Key words\/}: Differential variational inequality, stochastic variational inequality, parametric optimization, weak solution, measurability,  sample average approximation, time-stepping scheme.
\end{abstract}

\section{Introduction}

Let $X \subset \mathbb{R}^n$ be a nonempty closed convex set.
At each time $t \in \mathbb{R}_+$, uncertainty is represented
by a random vector $\xi:\Omega\to\mathbb{R}^{\ell}$, defined
on a probability space $(\Omega,\mathcal{F},\mathbb{P}_t)$,
with time-dependent distribution
$P_t:=\mathbb{P}_t\circ\xi^{-1}$ concentrated on a fixed
Borel set $\Xi\subseteq\mathbb{R}^{\ell}$.
Assume that there exists a fixed $\sigma$-finite Borel
measure $\mu$ on $\Xi$ such that
$
P_t\ll\mu
\ \text{for every } t\in\mathbb{R}_+.
$
We consider the following differential stochastic variational
inequality with parametric convex optimization (DSVI-O):
 \begin{equation}\label{DSTO_first_stage}
\left\{\begin{aligned}
\dot{x}(t) &= \Pi_X\left(x(t) - \mathbb{E}_{P_t}\left[\Phi\left(t, \xi, x(t), y(t, \xi)\right)\right]\right) - x(t),\quad t\ge 0,\\
x(0)&=x_0,
\end{aligned}\right.
\end{equation}
\begin{equation}\label{DSTO_second_stage}
\left\{\begin{aligned}
y_i(t, \xi) &\in \mathcal{S}_i(t, \xi, x(t)) := \arg\min_{\mathbf{y}_i \in \mathcal{Y}_i(t,\xi)} f_i(t, \xi, x(t), \mathbf{y}_i), \quad i = 1,\ldots,k,\\
y(t, \xi)&=\left(y_1(t, \xi),\ldots,y_k(t, \xi)\right), \quad {\rm a.e.}\,\,  \xi\in \Xi.
\end{aligned}\right.
\end{equation}
In (\ref{DSTO_first_stage}),
$\Phi : \mathbb{R}_+ \times \mathbb{R}^\ell \times \mathbb{R}^n
        \times \mathbb{R}^m \to \mathbb{R}^n$
is a continuous function and has the form
\begin{equation} \label{eq:special_case_dsto}
\Phi\left(t, \xi, x(t), y(t, \xi)\right)
= \Phi_1\left(t, \xi, x(t)\right) + \sum_{i=1}^k B_i(t, \xi,x(t))\, y_i(t, \xi),
\end{equation}
where $\Phi_1: \mathbb{R}_+ \times \mathbb{R}^\ell \times \bbr^n \to \mathbb{R}^n$ is a continuous function and   $B_i: \mathbb{R}_+ \times \mathbb{R}^\ell \times \bbr^n \to \mathbb{R}^{n \times m_i}$ is a continuous matrix-valued function for each $i \in [k]:= \{1,\dots,k\}$.
The operator $\Pi_X : \mathbb{R}^n \to X$ denotes the metric projection onto $X$.
To compute the function value of $\Phi$ at each time $t\ge 0$ and  event $\xi
\in \Xi$, we need to solve $k$ convex optimization problems in (\ref{DSTO_second_stage}), where for each $i\in [k]$, the objective function
$f_i : \mathbb{R}_+ \times \mathbb{R}^\ell \times \mathbb{R}^n
      \times \mathbb{R}^{m_i} \to \mathbb{R}$
is continuous, and convex in $\mathbf{y}_i$ for any fixed $(t,\xi, x(t))\in \mathbb{R}_+ \times \mathbb{R}^\ell \times \mathbb{R}^n,$ and the feasible set is a
set-valued map
$\mathcal{Y}_i : \mathbb{R}_+ \times \mathbb{R}^\ell
                 \rightrightarrows \mathbb{R}^{m_i}$
taking nonempty closed-convex values
$\mathcal{Y}_i(t,\xi)\subseteq \overline{\mathcal{Y}}_i \subset \mathbb{R}^{m_i}$ for all
$(t,\xi) \in \mathbb{R}_+ \times \mathbb{R}^\ell$.
Here $\overline{\mathcal{Y}}_i $ is a compact set for $i\in [k]$ and $m=\sum^k_{i=1}m_i$.

In many applications under dynamic and stochastic environments,
the parameters (mean, variance, etc.) of probability distributions change over time. Time-dependent distributions are more proper for
modelling non-stationary dynamic systems than static distributions
\cite{caozhangpoor21,cutlerdrusvyatskiyharchaoui23}.
In Section~5, we use the time-varying $P_t$ to study an embodied intelligence system for elderly care with heterogeneous data streams that exhibit temporal drift.

If $f_i$ is continuously differentiable with respect to $\y$ for each $i\in [k]$, then
(\ref{DSTO_second_stage}) can be rewritten as the following stochastic variational inequality
\begin{equation}\label{DSTO_second_VI}
0 \in \Psi(t, \xi, x(t),y(t, \xi)) +{\cal N}_{{\cal Y}(t,\xi)}(y(t,\xi)),
\end{equation}
where
$$\begin{aligned}
\Psi(t, \xi, x(t), y(t, \xi))= \left( \begin{array}{c}
\nabla_{\mathbf{y}_1} f_1(t, \xi, x(t), y_1(t, \xi))\\
\vdots\\
\nabla_{\mathbf{y}_k} f_k(t, \xi, x(t), y_k(t, \xi))
  \end{array}\right),
\end{aligned}
$$
${\cal Y}(t,\xi)={\cal Y}_1(t,\xi)\times\ldots \times {\cal Y}_k(t,\xi)$ and ${\cal N}_{{\cal Y}(t,\xi)}(y(t,\xi))={\cal N}_{{\cal Y}_1(t,\xi)}(y_1(t,\xi))\times\ldots \times {\cal N}_{{\cal Y}_k(t,\xi)}(y_k(t,\xi))$ is the normal cone of ${\cal Y}(t,\xi)$ at $y(t,\xi).$
When the distribution $P$ and feasible set ${\cal Y} $ are time-independent, then the DSVI-O (\ref{DSTO_first_stage})-(\ref{DSTO_second_stage}) reduces to the DSVI \cite{chen2022dynamic}
\begin{equation}\label{DSVI}
\tag{DSVI}
\left\{\begin{aligned}
\dot{x}(t) &= \Pi_X\left(x(t) - \mathbb{E}\left[\Phi\left(t, \xi, x(t), y(t, \xi)\right)\right]\right) - x(t),\quad t\ge 0,\\
x(0)&=x_0,\\
0&\in \Psi(t, \xi, x(t),y(t, \xi)) +{\cal N}_{{\cal Y}(\xi)}(y(t,\xi)), \quad {\rm a.e.}\,\, \xi\in \Xi.
\end{aligned}\right.
\end{equation}

The existence of a solution to the (DSVI) in the space of continuously differentiable functions and in
the space of measurable functions, as well as the convergence of the sample average approximation (SAA)
and the time-stepping approximation of the (DSVI) are proved in \cite{chen2022dynamic}.
These results are established under the assumption that \eqref{DSTO_second_VI} has a unique solution
$y_{\x}(t,\xi)$ for any $(t,\x)\in\bbr_+\times\bbr^n$ and for a.e.\ $\xi\in\Xi$.
Moreover, the function $(t,\x)\mapsto\E\!\left[\Phi\!\left(t,\xi,\x,y_{\x}(t,\xi)\right)\right]$ is assumed to be locally Lipschitz
continuous at any given $(t,\x)\in\bbr_+\times\bbr^n$.

If the (\ref{DSVI}) does not involve random variables, then it reduces  to the deterministic (DVI) \cite{camlibel2007lyapunov,chen2013,chen2014,ChenXiang,JSPang2010,PangDVI}
\begin{equation}\label{DVI}
\tag{DVI}
\left\{\begin{aligned}
\dot{x}(t) &= \Pi_X\left(x(t) - \Phi(t, x(t), y(t))\right) - x(t),\quad t\ge 0,\\
x(0)&=x_0,\\
0&\in \Psi(t,x(t),y(t)) +{\cal N}_{{\cal Y}}(y(t)).
\end{aligned}\right.
\end{equation}
If the data in (\ref{DSVI}) are time-independent and we seek
an equilibrium point at which $\dot{x}(t)=0$, then the system
reduces to the static two-stage stochastic variational
inequality (SVI) \cite{ChenLYZ,ChenTKW2017,ChenSS,RW}
\begin{equation}\label{SVI}
\tag{SVI}
\left\{\begin{aligned}
0 &\in \mathbb{E}\left[\Phi\left(\xi, x, y(\xi)\right)\right]+{\cal N}_{X}(x),\\
0 &\in \Psi(\xi, x,y(\xi))+ {\cal N}_{{\cal Y}(\xi)}(y(\xi)), \quad {\rm a.e.}\,\, \xi\in \Xi.
\end{aligned}\right.
\end{equation}

If $X=\bbr^n$, and the objective functions and constraints in (\ref{DSTO_second_stage}) are deterministic,  then (\ref{DSTO_first_stage})-(\ref{DSTO_second_stage}) reduces to the optimization-constrained ordinary differential equation (OCODE) \cite{LCH2009,LuoChen}
 \begin{equation}\label{DO}
 \tag{OCODE}
\left\{\begin{aligned}
\dot{x}(t) &= -\Phi(t, x(t), y(t)),\quad t\ge 0,\\
x(0)&=x_0,\\
y(t) &\in \mathop{\arg\min}_{\mathbf{y}\in \mathcal{Y}(t)} f(t, x(t), \mathbf{y}).
\end{aligned}\right.
\end{equation}

The models (\ref{DSVI}), (\ref{DVI}), (\ref{SVI}) and (\ref{DO}) have many applications in transportation sciences,
atmospheric chemistry, biological systems and economics and have been studied extensively over the last decades
(e.g., \cite{chen2022dynamic,PangDVI,RW}). Nowadays, Artificial Intelligence (AI) applications leverage static, dynamic, and random data in diverse forms,
including images, text, and audio~\cite{liu2025embodied} in time-varying environments.
 However, the models (\ref{DSVI}), (\ref{DVI}), (\ref{SVI}) and (\ref{DO}) have limitations in handling random data with time-dependent distributions from different sources in different forms simultaneously.
Motivated by existing research on  (\ref{DSVI}), (\ref{DVI}), (\ref{SVI}) and (\ref{DO}), and new applications in AI,
in this paper, we study the DSVI-O \eqref{DSTO_first_stage}--\eqref{DSTO_second_stage} where $k$
optimization problems in (\ref{DSTO_second_stage}) handle different learning problems with time-dependent distributions simultaneously. The objective functions in these optimization problems may include high-dimensional data and   nonsmooth regularizers, and thus are not necessarily strongly convex and differentiable.
Therefore, DSVI-O \eqref{DSTO_first_stage}--\eqref{DSTO_second_stage} cannot be written as the (\ref{DSVI}), (\ref{DVI}),
(\ref{SVI}) or (\ref{DO}) and existing theoretical results cannot be applied to obtain the existence of solutions of
DSVI-O \eqref{DSTO_first_stage}--\eqref{DSTO_second_stage} directly.

 We call a pair
$\left(x^*:[0,T] \to \bbr^n, y^*:[0,T]\times \Xi \to \bbr^{m} \right)$ a \emph{weak solution} of (\ref{DSTO_first_stage})-(\ref{DSTO_second_stage})
if $x^*$ is absolutely continuous in $[0,T]$ and
\begin{equation}\label{solu_x}
x^*(t) = x_0 + \int_0^t \Bigl(\Pi_X(x^*(\tau) - \mathbb{E}_{P_\tau}\left[ \Phi(\tau, \xi,x^*(\tau), y^*(\tau, \xi)) \right])-x^*(\tau) \Bigr) \, d\tau,
\end{equation}
and $y^*$ is jointly measurable in $(t,\xi)$ and satisfies
\[
\int_0^T \mathbb E_{P_t}[\|y^*(t,\xi)\|]\,dt<\infty,
\]
with $y_i^*(t,\xi)\in\mathcal S_i(t,\xi,x^*(t))$ for
$i\in[k]$, for a.e. $t\in[0,T]$ and a.e. $\xi$.

The main contributions of this paper are summarized as follows.
\begin{description}
\item{1.} We prove the existence of a weak solution of (\ref{DSTO_first_stage})-(\ref{DSTO_second_stage}) for any $T>0$ without assuming the differentiability of the objective functions $f_i(t,\xi,x(t), \cdot)$ and uniqueness of solutions $y_i(t,\xi)$ for $i\in [k].$  When $f_i$ is strongly convex in $\y$, we show that (\ref{DSTO_first_stage})-(\ref{DSTO_second_stage})
    has a classical solution $(x^*,y^*)$ with $x^*\in C^1[0,T]$ and $y^*(\cdot,\xi)$ is continuous in $[0,T]$ for a.e. $\xi \in \Xi$. Moreover,
         for any weak solution $(x^*(t), y^*(t,\xi))$, we show that $x^*(t)\in X$ for any $t\in[0,T]$ if $x_0\in X$.
    \item{2.} We propose a discrete approximation scheme for (\ref{DSTO_first_stage})-(\ref{DSTO_second_stage}), which  first uses the time-stepping method based on the forward Euler scheme over a uniform time grid to the ordinary differential equation (ODE) in \eqref{DSTO_first_stage}
        and then uses the SAA  to approximate the expectation value at each discrete time point.
        We give convergence theorems of the  discrete approximation scheme to a weak solution of (\ref{DSTO_first_stage})-(\ref{DSTO_second_stage}).
        \item{3.} We generate synthetic health care data by Multimodal Large Language Models, and test (\ref{DSTO_first_stage})-(\ref{DSTO_second_stage}) with the proposed discrete approximation scheme for problems in the embodied intelligence system for the elderly health. Preliminary results show the efficiency and effectiveness of the DSVI-O model.
\end{description}

\subsection{Preliminaries}

In this subsection, we introduce the notation and basic definitions used throughout the paper.

\paragraph{Notation}
Let $\|\cdot\|$ denote the Euclidean norm of a vector or a matrix. Let $\bbr_+$ denote the set of all nonnegative real numbers.
Given a convex set $X \subseteq \mathbb{R}^{n}$ and $\z\in\bbr^n$,  ${\rm dist}(\z, X) := \min_{\x \in X} \|\x - \z\|$ denotes the distance from $\z$ to $X$. Given two sets \( A, C \subseteq \mathbb{R}^n \),
$
\D(A, C) := \sup_{\x \in A} \operatorname{dist}(\x, C)
$
denotes the deviation of the set \( A \) from the set \( C \). The Hausdorff distance between $A$ and $C$ is defined by
$
\HH(A, C) := \max\big\{\D(A, C),\, \D(C, A)\big\}.
$
Given an $\x\in X$,
${\cal N}_X(\x)=\{\v\in\mathbb R^n:
\v^\top(\z-\x)\le0,\ \forall\z\in X\}$ denotes the normal cone of $X$ at $\x$,
and
${\cal T}_X(\x)=\{\y\in\mathbb R^n:
\liminf_{\lambda\to0+}\lambda^{-1}\dist
(\x+\lambda\y,X)=0\}$ denotes the contingent cone to $X$ at $\x$.
Let $C_b(X)$ denote the set of real-valued bounded continuous functions on $X$, and  $\operatorname{conv}(C)$ denote the convex hull of a set $C$.

\begin{definition}[Upper/lower semicontinuous {\rm \cite{aubin2009differential}}]
     We say that a set-valued map $G: \mathbb{R}^{n} \rightrightarrows \mathbb{R}^{m}$ is {upper semicontinuous} (u.s.c.) if for any \( \x_0 \in \mathbb{R}^{n} \), for any open set \( \mathcal{D} \) containing \( G(\x_0) \), there exists a neighborhood \( \mathcal{N}_{\x_0} \) of \( \x_0 \) such that for any $\x\in \mathcal{N}_{\x_0} $, \( G(\x) \subset \mathcal{D} \).
    Similarly, we say that \( G\) is {lower semicontinuous} (l.s.c.) if for any \( \x_0 \in \mathbb{R}^{n} \), for any \( \y_0 \in G(\x_0) \) and any neighborhood \( \mathcal{N}_{\y_0} \) of  \( \y_0 \), there exists a neighborhood \( \mathcal{N}_{\x_0} \) of \( \x_0 \) such that  for any
$\x \in \mathcal{N}_{\x_0}, \ G(\x) \cap \mathcal{N}_{\y_0} \neq \emptyset.
$  Moreover, $G$ is said to be continuous if and only if it is both u.s.c. and l.s.c.
\end{definition}

\begin{definition}[Measurability {\rm \cite{aubin1999set}}]
Consider a measurable space $(\mathbb{R}^{m}, \mathcal{B}({\mathbb{R}^{m}}))$ and a set-valued map $G : \mathbb{R}^{m} \rightrightarrows \mathbb{R}^{n}$ with closed images. The map $G$ is called \emph{measurable} if for every open subset $\mathcal{O} \subset \mathbb{R}^{n}$, we have
$
G^{-1}(\mathcal{O}) := \{\omega \in \mathbb{R}^{m} : G(\omega) \cap \mathcal{O} \neq \emptyset\} \in \mathcal{B}({\mathbb{R}^{m}}).
$
\end{definition}

\begin{definition}[Random closed set {\cite{molchanov2005theory}}] For a probability space $(\Omega, \mathcal{F}, \mathbb{P})$,
a set-valued map $\zeta: \Omega \rightrightarrows {\mathbb{R}^{n}}$ with closed values is called a \emph{random closed set} if, for every compact set $\mathcal{O} \subseteq \mathbb{R}^{n}$,
$
\{\omega \in \Omega : \zeta(\omega) \cap \mathcal{O} \neq \emptyset\}\in \mathcal{F}
$.
\end{definition}

\begin{definition}[Integrably bounded {\rm\cite{molchanov2005theory}}]\label{def_integrably bounded}
 A random closed set \( \zeta \) is said to be \emph{integrably bounded} if
 $
  \|\zeta\| := \sup\{\, \|\varsigma\| : \varsigma \in \zeta \,\}
$
  has a finite expectation, i.e., \( \mathbb{E}[\|\zeta\|] < \infty \).
\end{definition}

\begin{definition}[Integrable selections {\rm\cite{molchanov2005theory}}]
Let $\zeta:\Omega\rightrightarrows \mathbb{R}^{n}$ be a random closed set. The family of integrable selections of \( \zeta \) is defined as
$
L^1(\zeta) := L^0(\zeta) \cap L^1(\mathbb{R}^{n}),
$
where \( L^0(\zeta) \) denotes the family of all measurable selections of \( \zeta \), and \( L^1(\mathbb{R}^{n}) \) is the space of measurable random elements $\xi$ with $\mathbb{E}[\|\xi\|]<\infty$.
\end{definition}

The rest of this paper is organized as follows. In Section 2, we prove the existence of a weak solution of    (\ref{DSTO_first_stage})-(\ref{DSTO_second_stage}) and the feasibility of any weak solution    of (\ref{DSTO_first_stage})-(\ref{DSTO_second_stage}). In Section 3, we study the discrete scheme for  (\ref{DSTO_first_stage})-(\ref{DSTO_second_stage}) and prove its convergence.  In Section 4, we illustrate the theoretical results on the existence of solutions and discrete scheme with numerical examples.    In Section 5, we study applications of DSVI-O in the embodied intelligence system for the elderly health by synthetic health care data generated by Multimodal Large Language Models. In Section 6, we present concluding remarks.

\section{Existence of weak solutions of (\ref{DSTO_first_stage})–(\ref{DSTO_second_stage})}\label{exist}
The right-hand side of the ODE  in (\ref{DSTO_first_stage}) involves $k$ solutions $y_i(t,\xi), i\in [k]$ of the $k$ optimization problems in (\ref{DSTO_second_stage}) for $t>0$ and a.e. $\xi\in\Xi$. Denote
the joint solution set of the optimization problems by
\[
\mathcal{S}(t,\xi,x(t)) \;=\; \mathcal{S}_1(t,\xi,x(t)) \times \mathcal{S}_2(t,\xi,x(t)) \times \cdots \times \mathcal{S}_k(t,\xi,x(t)).
\]
For $T>0$, define a set-valued map $F:[0,T]\times \bbr^n
\rightrightarrows \bbr^n$ as
\begin{equation}\label{F_t_x}
F(t,x(t)) = \mathbb{E}_{P_t}\big\{\, \Phi(t,\xi,x(t),y(t,\xi)) \;:\; y(t,\xi) \in \mathcal{S}(t,\xi,x(t)) \,\big\}.
\end{equation}
Then \eqref{DSTO_first_stage}
can be rewritten as:
\begin{equation}\label{ODE1}
\dot{x}(t) \;\in\;H(t,x(t)):= \Big\{\,\Pi_X\!\big(x(t) - \varphi\big) - x(t) : \varphi \in F(t,x(t)) \,\Big\}.
\end{equation}

In this section, we show that
the  linear combination structure of $y_i$ in \eqref{eq:special_case_dsto} preserves the convexity  of the set-valued map $F$, which  is commonly used in   differential variational inequality (DVI) and control works (e.g., \cite{BacciottiRosier2005,donchev1998stability,PangDVI}).
If $\Phi$ depends on $\y$ in a general nonlinear structure, $F$ may become nonconvex-valued,
and existence results need additional structural conditions (see, e.g., \cite{Bounkhel2002}) or conditions on $\mathcal{S}(t,\xi,x(t))$ \cite[A.2]{chen2022dynamic}. These additional conditions are
too strong for some applications, when stochastic optimization problems for data fitting in (\ref{DSTO_second_stage}) have multiple solutions.

Due to the selection of solutions from stochastic optimization problems in (\ref{DSTO_second_stage}), we cannot expect that (\ref{DSTO_first_stage}) has a differentiable solution $x$ for any $T>0$ without strong reasonable conditions.  One may ask if we can
select the least-norm solution
\[\bar{y}(t,\xi)= \arg\min_{\y\in \mathcal{S}(t,\xi,x(t))} \|\y\|^2.\]
Since the solution set is convex,  there is a unique $\bar{y}(t,\xi)$ for any
$t\in [0, T],$ a.e. $ \xi\in \Xi.$ However, even if  \( \mathcal{S}(t, \xi, \x) \) is Lipschitz continuous as a set-valued map in Hausdorff distance, the least-norm solution $\bar{y}(t,\xi)$ may fail to be Lipschitz continuous in $(t,\x)\in[0,T]\times \bbr^n$, as illustrated by \cite[Counterexample, pp. 70-72]{aubin2009differential}.
Moreover, the solution set \( \mathcal{S}(t, \xi, \x) \) is rarely Lipschitz continuous in general. Under typical assumptions, one can at most expect upper semicontinuity, which may not even admit a continuous selection  \cite[pp.~39]{aubin2009differential}.

In this section, we focus on the  existence of a weak solution of the DSVI-O  (\ref{DSTO_first_stage})–(\ref{DSTO_second_stage}), which allows us to work with integrable and measurable selections while preserving the meaningful evolution of the system state \( x(t) \).

For $x_0\in X$ and $T>0$, we  make the following assumptions.
\begin{enumerate}[label=\textbf{(A\arabic*)}]
\item \label{A1}
$\mathbb{P}_{(\cdot)}(\cdot) : [0,T] \times \mathcal{F} \to [0,1]$ is a \emph{probability kernel} from $[0,T]$ to $\Omega$, that is, for each $A \in \mathcal{F}$, the mapping $t \mapsto \mathbb{P}_t(A)$ is $\mathcal{B}_{[0,T]}$ measurable, and for each $t \in [0,T]$, $\mathbb{P}_t(\cdot)$ is a probability measure on $(\Omega, \mathcal{F})$ \cite{kallenberg1997foundations}.

  \item \label{A2}
There exist measurable functions \(\kappa_{\Phi_1}, \kappa_{B_i} : [0,T] \times \Xi \to \mathbb{R}_+\), \(i \in[k]\), such that their expectations \(\bar{\kappa}_{\Phi_1}(t) := \mathbb{E}_{P_t}[\kappa_{\Phi_1}(t,\xi)]\) and \(\bar{\kappa}_{B_i}(t) := \mathbb{E}_{P_t}[\kappa_{B_i}(t,\xi)]\) are finite for every $t\in[0,T]$ and integrable over $[0,T]$, and for all \(\mathbf{x}, \mathbf{x}' \in \mathbb{R}^n\),
\[
\|\Phi_1(t,\xi,\mathbf{x}) - \Phi_1(t,\xi,\mathbf{x}')\| \le \kappa_{\Phi_1}(t,\xi) \|\mathbf{x} - \mathbf{x}'\|, \]
and
\[
\|B_i(t,\xi,\mathbf{x}) - B_i(t,\xi,\mathbf{x}')\| \le \kappa_{B_i}(t,\xi) \|\mathbf{x} -\mathbf{x}'\|.
\]

\item \label{A3} There exists a constant \(\bar{\beta} > 0\) such that, for \(i\in [k]\),
\[
\sup_{t\in[0,T]}\max\left\{
\mathbb E_{P_t}[\|\Phi_1(t,\xi,x_0)\|],
\mathbb E_{P_t}[\|B_i(t,\xi,x_0)\|]
\right\}\le\bar\beta,\quad i\in[k].
\]
  \item \label{A4} The set-valued map $\mathcal{Y}:[0,T]\times \Xi\rightrightarrows \bar{\cal Y}:= \bar{\cal Y}_1\times \cdots \times \bar{\cal Y}_k\subset \bbr^m$ is continuous.
\end{enumerate}

\begin{remark}
Assumption \ref{A1} requires only that $\mathbb{P}_t$ is a probability kernel from $[0,T]$ to $\Omega$, without imposing any continuity or smoothness in time.  This condition ensures that the random environment is in a measurably well-defined way, and enables us to prove the measurability of the set-valued map $F$ in the right-hand side of (\ref{ODE1}) at $t\in[0,T]$.
Assumptions \ref{A2}-\ref{A3} impose Lipschitz continuity conditions on the mappings $\Phi_1$ and $B_i$ in $x$, with integrable expected Lipschitz moduli, supporting the boundedness and linear-growth of the right-hand side of (\ref{ODE1}).
Finally, Assumption \ref{A4} guarantees that the solution set map ${\cal S}$ is u.s.c. and takes
nonempty and compact values.
\end{remark}

Let $M :=\max_{i\in[k]}\;\sup_{\mathbf{y} \in \overline{\mathcal{Y}}_{i}} \, \|\mathbf{y}\|$. Before establishing the existence of weak solutions, we first present properties of the set-valued maps ${\cal S}$ and $F$.

\begin{lemma}\label{lemma1}
Under Assumptions \ref{A1}–\ref{A4}, the following statements hold.
\begin{description}
\item{(i)}  ${\cal S}(t,\xi,\mathbf{x})$ is
nonempty, convex and compact for any fixed $t\geq0$, $\xi \in \Xi$  and $\mathbf{x}\in\mathbb{R}^n$, and ${\cal S}$ is u.s.c. ;
\item{(ii)}
$F(t,\mathbf{x})$ is nonempty, convex and compact for any fixed $t\geq0$ and $\mathbf{x}\in\mathbb{R}^n$,
$F(t,\cdot)$ is u.s.c. for any $t\geq0$ and $F(\cdot,\x)$ is measurable for any fixed $\mathbf{x}\in\mathbb{R}^n$.
\end{description}
\end{lemma}

\begin{proof}
(i) For fixed $t \in \mathbb{R}_+$, $\xi \in \Xi$, and $\mathbf{x} \in \mathbb{R}^n$, since the feasible set $\mathcal{Y}_i(t,\xi)$ is closed and contained in a compact set $\overline{\mathcal{Y}}_i$, and the objective $f_i$ is continuous in $\y_i$, by the Weierstrass theorem, one knows that each set-valued map $\mathcal{S}_i$ is nonempty-valued and compact-valued, and so is the joint map $\mathcal{S}$. Also, noting Assumption \ref{A4},  it then from \( f_i \) is continuous on \( \mathbb{R}_+ \times \Xi \times \mathbb{R}^n \times \mathbb{R}^{m_i} \), $f_i(t,\xi,\mathbf{x},\cdot)$ is convex on \( \mathbb{R}^{m_i} \) and \cite[Theorem 3.1]{terazono2015continuity} that the set-valued map \( \mathcal{S}_i \) is convex-valued and u.s.c. on \( \mathbb{R}_+ \times \Xi \times X \). Hence, it follows from \cite[Exercise 2.20]{rockafellar2009variational} that \( \mathcal{S} \) is convex-valued. Noticing the compactness of \( \mathcal{S} \), we conclude from \cite[Proposition 4 pp.64 and Theorem 2 pp.62]{aubin2009differential} that \( \mathcal{S} \) is u.s.c.

(ii) For any \( (t, \mathbf{x}) \in \mathbb{R}_+ \times \mathbb{R}^n \) and a given realization \( \xi \in \Xi \), define the set-valued map
\begin{equation}\label{G_fun}
\mathcal{G}(t,\xi,\mathbf{x}) := \big\{\, \Phi(t,\xi,\mathbf{x},\y) \;:\; \y \in \mathcal{S}(t,\xi,\mathbf{x}) \,\big\}.
\end{equation}
Since \( \mathcal{S} \) is u.s.c. with nonempty closed values, it is measurable by \cite[Proposition 8.2.1]{aubin1999set}. Moreover, since \( \Phi \) is continuous, it follows from \cite[Lemma 8.2.8]{aubin1999set} that \( \mathcal{G} \) is measurable. In addition, for each fixed \( (t,\xi,\mathbf{x}) \), the set \( \mathcal{S}(t,\xi,\mathbf{x}) \) is compact and \( \Phi(t,\xi,\mathbf{x},\cdot) \) is continuous, which implies that \( \mathcal{G}(t,\xi,\mathbf{x}) \) is compact. Thus, it follows from the u.s.c. property of $\mathcal{S}$, the continuity of $\Phi$, and \cite[Proposition 1.4.14]{aubin1999set} that $\mathcal{G}$ is u.s.c.

We now study the properties of $F$. For fixed $(t,\mathbf{x})\in \mathbb{R}_+ \times \mathbb{R}^n$, $F(t,\mathbf{x})$ is defined as the Aumann integral of the random set $\mathcal{G}(t,\xi,\mathbf{x})$ with respect to the time-varying probability measure $\mathbb P_t$
\[
\begin{aligned}
F(t,\mathbf{x})&=\int \mathcal{G}(t,\xi,\mathbf{x})\, d{P}_t(\xi)= \Big\{\,\mathbb{E}_{{P}_t}[\,g(\xi)\,] :   g(\xi)\in L^1(\mathcal{G}(t,\xi,\mathbf{x}))\Big\}.
\end{aligned}
\]
By Assumptions \ref{A2} and \ref{A3} and recalling Definition \ref{def_integrably bounded}, for fixed $(t,\mathbf{x})$, it follows

\begin{equation}\label{upper_F_t_x}
\begin{aligned}
&\mathbb{E}_{P_t}\!\Bigl[\sup_{\mathbf{z}\in\mathcal{G}(t,\xi,\x)}\|\mathbf{z}\|\Bigr]
=\mathbb{E}_{P_t}\!\Bigl[
      \sup_{\mathbf{y}\in\mathcal{S}(t,\xi,\x)}
      \Bigl\|\Phi_1(t,\xi,\mathbf{x})+\sum_{i=1}^k B_i(t,\xi,\mathbf{x})\,\mathbf{y}_i\Bigr\|
    \Bigr] \\
&\le {\mathbb{E}_{P_t}[\|\Phi_1(t,\xi,x_0)\|]}
     +\|\mathbf{x}-x_0\|\,\bar{\kappa}_{\Phi_1}(t) \\[4pt]
&\quad
     +M\sum_{i=1}^{k}
       \Bigl(
         {\mathbb{E}_{P_t}[\|B_i(t,\xi,x_0)\|]}
         +\|\mathbf{x}-x_0\|\,\bar{\kappa}_{B_i}(t)
       \Bigr) \\
&\le (kM+1)\bar{\beta}
     +\bigl(\bar{\kappa}_{\Phi_1}(t)+M\!\sum_{i=1}^k\bar{\kappa}_{B_i}(t)\bigr)\|\mathbf{x}-x_0\|< \infty,
\end{aligned}
\end{equation}
which implies $\mathcal{G}(t,\xi,\mathbf{x})$ is integrably bounded.  Hence, given its measurability as established, \cite[Theorem 8.1.3]{aubin1999set} ensures that $L^1(\mathcal{G}(t,\xi,\mathbf{x})) \neq \emptyset$ for fixed $(t,\mathbf{x})$, and thus $F(t,\mathbf{x})$ is nonempty. In addition, since $\Phi(t,\xi,\mathbf{x},\cdot)$ is affine and $\mathcal{S}$ is convex-valued, the mapping $\mathcal{G}$ is convex-valued. Thus, Since $\mathcal G$ takes values in the finite-dimensional space
$\mathbb R^n$,  it follows from \cite[Theorems 2.1.30 and 2.1.37]{molchanov2005theory} that the set-valued map $F$ takes convex and closed values. Hence, $F(t,\cdot)$ is compact-valued.

Moreover, since for each fixed $(t,\mathbf{x})\in \mathbb{R}_+\times \mathbb{R}^n$, $\mathcal{G}(t,\cdot,\mathbf{x})$ is measurable, by \cite[Theorem 8.1.4]{aubin1999set}, $\mathcal{G}(t,\cdot,\mathbf{x}):\xi\rightrightarrows {\mathbb{R}^n}$ has a measurable graph. Therefore, noting that $\mathcal{G}$ is u.s.c., \cite[Theorem 3.1 and Remark 3.1]{yannelis1990upper} ensures that for fixed $t\in\mathbb{R}_+$, the map $F(t,\cdot)$ is u.s.c.
Now, we fix $\mathbf{x}\in\mathbb{R}^n$. Note that
$\mathcal{G}(\cdot,\cdot,\mathbf{x})$ is a measurable set-valued map on $[0,T]\times\Xi$
with compact values in $\mathbb{R}^n$, and \eqref{upper_F_t_x} gives the $P_t$-integrable boundedness
of $\mathcal{G}(t,\cdot,\mathbf{x})$ for each $t\in[0,T]$. Hence, by \cite[Corollary~4.7]{artstein1989parametrized}, the map
$F(\cdot,\mathbf{x})
$
is measurable.
\end{proof}

\begin{theorem}[Existence of a weak solution]\label{thm:existence}
  Under Assumptions \ref{A1}–\ref{A4}, if $n=1$ or
  $X=\{\x\in \mathbb{R}^n:A\x=\mathbf{b}\}$ for some given $A\in\mathbb{R}^{m\times n}$ and $\mathbf{b}\in\mathbb{R}^m$, then for any $T>0$,
  the DSVI-O (\ref{DSTO_first_stage})–(\ref{DSTO_second_stage}) has at least one weak solution $\left(x^*,y^*\right)$.
\end{theorem}

\begin{proof}

We now investigate the differential equation \eqref{DSTO_first_stage}.
Note that \eqref{DSTO_first_stage}
with $y_i(t,\xi) \in \mathcal{S}_i(t,\xi,x(t))$ for $i\in [k]$, can be rewritten using $F(t,x(t))$ as (\ref{ODE1}). Recall the set-valued map
$
H(t,\mathbf{x}) := \{\Pi_X(\mathbf{x} - \varphi) - \mathbf{x} : \varphi \in F(t,\mathbf{x})\}.
$
We next verify the assumptions of \cite[Theorem~5.2]{deimling2011multivalued}.

{\it Compact convex-valuedness of $H$.} Since $F(t,\x)$ takes nonempty and compact values and, for each $(t,\x)$, the map
$\varphi\mapsto \Pi_X(\x-\varphi)-\x$ is continuous, it follows that $H(t,\x)$ is nonempty and compact.
If $X=\{\x:A\x=\mathbf{b}\}$ with $A\in\mathbb{R}^{m\times n}$ (and $X\neq\emptyset$), using the explicit formula
$
\Pi_X(\z)=\z-A^{\dagger}(A\z-\mathbf{b}),
$
where $A^{\dagger}$ denotes the Moore--Penrose pseudoinverse of $A$, we get
$
\Pi_X(\x-\varphi)-\x
= -A^{\dagger}(A\x-\mathbf{b})\;-\;\big(I-A^{\dagger}A\big)\varphi,
$
which is affine in $\varphi$ and therefore maps the convex set $F(t,\x)$ to a convex set.
If $n=1$, then both $X$ and $F(t,\x)$ are intervals, and since the one-dimensional projection $\Pi_X$ maps intervals to intervals, $H(t,\x)$ is an interval and therefore convex.

{\it Measurability of $H(\cdot,\x)$ and upper semicontinuity of $H(t,\cdot)$.} By Lemma \ref{lemma1}(ii), we have \( F(\cdot,\mathbf{x}) \) is measurable for each fixed \( \mathbf{x} \). Noting that the \( \Pi_X(\cdot)\) is Lipschitz continuous, it follows that \( H(\cdot, \mathbf{x}) \) is measurable. By \cite[Proposition 1.4.14]{aubin1999set}, the upper semicontinuity of \( H(t,\mathbf{x}) \) in \( \mathbf{x} \) for each fixed \( t \) follows from the upper semicontinuity of \( F(t,\mathbf{x}) \) and the continuity of the map \( (\mathbf{x},\varphi) \mapsto \Pi_X(\mathbf{x} - \varphi) - \mathbf{x} \).

{\it  Linear growth of $H$.} In addition, for any $T>0$, by the non-expansiveness of \( \Pi_X \) and noting (\ref{upper_F_t_x}) and $x_0\in X$, it follows
\begin{equation}\label{linear_growth}
\begin{aligned}
\sup_{\mathbf{v}\in H(t,\mathbf{x})}\|\mathbf{v}\|
&= \sup_{\varphi\in F(t,\mathbf{x})}\|\Pi_X(\mathbf{x}-\varphi)-\mathbf{x}\|\\
&=\sup_{\varphi\in F(t,\mathbf{x})}\|\Pi_X(\mathbf{x}-\varphi)-\Pi_X(\mathbf{x})+\Pi_X(\mathbf{x})-\mathbf{x}\|\\
&\le \sup_{\varphi\in F(t,\mathbf{x})}\|\varphi\| + \|\x - \Pi_X(\mathbf{x})\| \\
&\le (kM+1)\bar{\beta}
     + \bigl(\bar{\kappa}_{\Phi_1}(t) + M\textstyle\sum_{i=1}^k \bar{\kappa}_{B_i}(t) + 1\bigr)\|\mathbf{x} - x_0\| \\
&\le \widetilde{L}_1(t)\|\mathbf{x}\| + \widetilde{L}_2(t),
\end{aligned}
\end{equation}
where
$
\widetilde{L}_1(t)
:= \bar{\kappa}_{\Phi_1}(t) + M\textstyle\sum_{i=1}^k \bar{\kappa}_{B_i}(t) + 1$ and $
\widetilde{L}_2(t)
:= (kM+1)\bar{\beta} + \widetilde{L}_1(t)\|x_0\|.
$
Since $\bar{\kappa}_{\Phi_1}$ and $\bar{\kappa}_{B_i}$ are integrable over the interval [0,T], so are
$\widetilde{L}_1$ and $\widetilde{L}_2$. Thus \(H(t,\mathbf{x})\) satisfies the linear-growth condition.

{\it Tangency condition}.
Finally, for fixed
$(t,\mathbf{x})\in[0,T]\times X$, pick $\varphi\in F(t,\mathbf{x})$ and set
$
\v=\Pi_X\bigl(\mathbf{x}-\varphi\bigr)-\mathbf{x}\in H(t,\mathbf{x}).
$ Since $\Pi_X(\mathbf{x}-\varphi)\in X$ and $X$ is convex,
the segment $\mathbf{x}+\lambda \v=(1-\lambda)\mathbf{x}+\lambda \Pi_X(\mathbf{x}-\varphi)$ lies in $X$
for every $\lambda\in(0,1]$.  Hence, $\frac{\x+\lambda \v -\x}{\lambda}=\v$, which means $\v$ is in the tangent cone to $X$ at $\x$, and thus
$
\v\;\in\;H(t,\mathbf{x})\cap \mathcal{T}_X(\mathbf{x})\neq\emptyset
$ for any $(t,\mathbf{x})\in[0,T]\times X$.

Therefore, according to \cite[Theorem~5.2]{deimling2011multivalued}, for any given initial state $x(0) = x_0 \in X$, there exists at least one absolutely continuous trajectory $x^*(t)$ to (\ref{ODE1}) on $[0,T]$ such that $x^*(0)=x_0$.

At last, we construct a jointly measurable selection $y^*$ for the
optimization problems in \eqref{DSTO_second_stage} that is compatible
with the trajectory $x^*$. Let $v^*$ be a measurable version of
$\dot{x}^*$, and, for a.e. $t\in[0,T]$, define
$\mathcal Q(t):=\{\varphi\in F(t,x^*(t)):
v^*(t)=\Pi_X(x^*(t)-\varphi)-x^*(t)\}$.
Since $x^*$ satisfies the differential inclusion, $\mathcal Q(t)$ is
nonempty for a.e. $t\in[0,T]$. 
The set-valued map $(t,\xi)\mapsto
\mathcal G(t,\xi,x^*(t))$ is jointly measurable.
Applying the parametrized-integral measurability result
\cite{artstein1989parametrized} to this set-valued map shows that
$t\mapsto F(t,x^*(t))$ is measurable.
Consequently, $\mathcal Q$ has a measurable graph
and compact values.
Thus, by
\cite[Theorem~8.1.3]{aubin1999set}, there exists a measurable
selection $\varphi^*(t)\in\mathcal Q(t)$.
Since
$\varphi^*(t)\in F(t,x^*(t))
=\int_{\Xi}\mathcal G(t,\xi,x^*(t))\,P_t(d\xi)$,
it follows from \cite[Corollary~4.7]{artstein1989parametrized} that
there exists a jointly measurable selection
$g^*(t,\xi)\in\mathcal G(t,\xi,x^*(t))$ such that
$\varphi^*(t)=\int_{\Xi}g^*(t,\xi)\,P_t(d\xi)$ for a.e.
$t\in[0,T]$. Define
$\mathcal R(t,\xi):=\{\y\in\mathcal S(t,\xi,x^*(t)):
\Phi(t,\xi,x^*(t),\y)=g^*(t,\xi)\}$.
By the definition of $\mathcal G$, $\mathcal R$ has nonempty compact
values and a measurable graph. Hence, by
\cite[Theorem~8.1.3]{aubin1999set}, there exists a jointly measurable
selection $y^*(t,\xi)\in\mathcal R(t,\xi)$. Therefore,
$\varphi^*(t)=\mathbb E_{P_t}[
\Phi(t,\xi,x^*(t),y^*(t,\xi))]$ for a.e. $t\in[0,T]$.
Moreover, $\|y^*(t,\xi)\|\le\sqrt{k}M$, and hence
$\int_0^T\mathbb E_{P_t}[\|y^*(t,\xi)\|]\,dt\le\sqrt{k}MT<\infty$.
Together with the state trajectory $x^*$, the selection $y^*$
satisfies, for every $t\in[0,T]$,
\[
x^{*}(t)
=
x_{0}
+
\int_{0}^{t}
\left(
\Pi_{X}\!\left(
x^{*}(\tau)
-
\mathbb{E}_{P_{\tau}}\!\left[
\Phi\bigl(\tau,\xi,x^*(\tau),y^*(\tau,\xi)\bigr)
\right]
\right)
-
x^{*}(\tau)
\right)
\,d\tau .
\]
We complete the proof.

\end{proof}

\begin{remark}\label{rem:affineX_weak}
The additional condition on $X$ in Theorem~\ref{thm:existence} is used to keep the induced set-valued map $H$ in (\ref{ODE1}) convex-valued. Without this condition, the projection $\Pi_X(\cdot)$ cannot preserve convexity for a convex set-valued map, and \cite[Theorem~5.2]{deimling2011multivalued} is invalid.  For  nonconvex differential inclusions, existence results require additional structural conditions (see, e.g., \cite{Bounkhel2002}), which are hard to verify here since the induced map is governed by the parametric second-stage problem \eqref{DSTO_second_stage}.

To relax this condition to a general convex set  $X$, following \cite{chen2022dynamic}, we can use an alternative assumption that for any given $(t,\x)$ there exists a measurable selection
$y_{\x}(t,\xi)\in\mathcal{S}(t,\xi,\x)$ such that the induced right-hand-side term
$
G(t,\x)\;:=\;\mathbb{E}_{P_t}\!\left[\Phi\big(t,\xi,\x,y_\x(t,\xi)\big)\right]
$
is (locally) Lipschitz continuous at $(t,\x)$.
If $f_i(t,\xi,\x,\cdot)$ is uniformly strongly convex and 
$\nabla_{y_i}f_i(t,\xi,\x,y_i)$ is uniformly Lipschitz continuous with respect to $(t,\x)$, then there exists such 
measurable selection. 

\end{remark}

Now we consider a special case of the DSVI-O \eqref{DSTO_first_stage}-\eqref{DSTO_second_stage} where the objective function $f_i$ is uniformly strongly convex in $\y_i$ for $i\in [k]$ in the sense that there is $\eta_i>0$ such that for any $\lambda\in (0,1)$,
\begin{equation}\label{strong}
f_i(t,\xi,\x,\lambda\y+(1-\lambda)\z)
\le \lambda f_i(t,\xi,\x,\y)+(1-\lambda)f_i(t,\xi,\x,\z)-\eta_i\lambda(1-\lambda)\|\y-\z\|^2
\end{equation}
holds for every $(t,\xi,\x)\in[0,T]\times\Xi\times\mathbb R^n$
and all $\y,\z\in\mathbb R^{m_i}$.

\begin{corollary}[Existence of a classical solution]\label{coro:existence}
Suppose that $f_i$ satisfy (\ref{strong}) for $i\in[k]$, and that
$\Phi_1(t,\xi,x_0)$, $B_i(t,\xi,x_0)$,
$\kappa_{\Phi_1}(t,\xi)$, and $\kappa_{B_i}(t,\xi)$
have uniformly bounded $p$-th moments under $P_t$ on $[0,T]$
for some $p>1$ and $i\in[k]$.
Assume also that $P_t$ is a Feller kernel \cite{beneduci2016positive}, i.e., the function $t\mapsto \bbe_{P_t}[\psi(\xi)]$ is continuous and bounded in $[0,T]$  whenever $\psi\in C_b(\Xi)$. Then under Assumptions \ref{A1}–\ref{A4},
  the DSVI-O (\ref{DSTO_first_stage})–(\ref{DSTO_second_stage}) has at least one classical solution $\left(x^*,y^*\right)$ for any $T>0$ in the sense that $(x^*,y^*)$ is a weak solution of (\ref{DSTO_first_stage})–(\ref{DSTO_second_stage}) with $x^*$ being continuously differentiable and $y^*(\cdot, \xi)$ being continuous for every $\xi\in\Xi$.
\end{corollary}
\begin{proof}
For each $i\in[k]$, since $f_i$, $i\in[k]$ is strongly convex~\eqref{strong}, there exists a unique solution $\hat y_i(t,\xi,\x)$ of \eqref{DSTO_second_stage} for any $(t,\xi,\x)\in\bbr_+\times \Xi\times X$.
Moreover, by following the proof of Lemma~\ref{lemma1}, we obtain that the joint solution map \(\mathcal{S}\) is u.s.c. Since it is single-valued, we conclude that $\hat y(t, \xi, \mathbf{x}):=(\hat y_1(t,\xi,\x),\ldots,\hat y_k(t,\xi,\x))$ is continuous on \(\bbr_+ \times \Xi \times X\).
From Lemma~\ref{lemma1}, the function \( F(t,\x) = \mathbb{E}_{P_t}\left[\Phi(t,\xi,\x,\hat y(t, \xi, \x))\right] \) is well-defined.
Now we prove the continuity of \(F\). Let \((t_\nu, \x_\nu) \to (t,\x)\) be any convergent sequence in \([0,T] \times X\). For each \(\nu \in \mathbb{N}\), define
$
\psi_\nu(\xi) := \Phi(t_\nu,\xi,\x_\nu,\hat y(t_\nu, \xi,\x_\nu))$ and $
\psi(\xi) := \Phi(t,\xi,\x,\hat y(t,\xi,\x)).
$

Since $P_t$ is a Feller kernel, we have
$
P_{t_\nu}\Rightarrow P_t
\ \text{as } \nu\to\infty.
$
Moreover, the joint continuity of $\Phi$ and $\hat y$ implies that
$
\lim_{\substack{\nu\to\infty,\xi'\to\xi}}
\psi_\nu(\xi')
=
\psi(\xi),
\ \xi\in\Xi.
$
By \eqref{upper_F_t_x}, the boundedness of
$\{\x_\nu\}_{\nu\ge1}\cup\{\x\}$, and the uniformly bounded
$p$-th moments, we have
\[
\sup_{\nu\ge1}
\mathbb E_{P_{t_\nu}}
\big[\|\psi_\nu(\xi)\|^p\big]
+
\mathbb E_{P_t}
\big[\|\psi(\xi)\|^p\big]
<\infty.
\]
Since $p>1$, the sequence $\{\psi_\nu\}$ is uniformly integrable
with respect to $\{P_{t_\nu}\}$. Therefore, by the Lebesgue's dominated convergence theorem for weakly converging measures, applied componentwise,
\[
F(t_\nu,\x_\nu)
=
\mathbb E_{P_{t_\nu}}[\psi_\nu]
\longrightarrow
\mathbb E_{P_t}[\psi]
=
F(t,\x).
\]
Thus, $F$ is continuous on $[0,T]\times X$.

Moreover, since $\Pi_X(\cdot)$ is Lipschitz continuous with Lipschitz constant one, the right-hand side of the ODE
\begin{equation}\label{ODE_H}
\dot{x}(t)=\Pi_{X}(x(t)-F(t,x(t)))-x(t)
\end{equation}
is continuous in $[0,T]\times X$. Thus, the linear growth condition (\ref{linear_growth}) and the classical Peano and extension theorem (e.g.\ \cite[Theorems 1.1 and 2.1]{hale2009ordinary}) imply that the ODE \eqref{ODE_H} has at least a solution
$x^*\in C^{1}([0,T])$.
Set
$
      y^*(t,\xi):=\hat y(t,\xi,x^*(t))$ for any $t\in[0,T]$ and $\xi\in\Xi.$
Since \( \hat{y} \) is continuous and \( x^* \in C^1([0,T]) \), it follows that the mapping \( t \mapsto y^*(t,\xi) \) is continuous for every $\xi\in\Xi$.
Moreover, as \( \hat{y}(t, \xi, \x) \) is bounded, we have \( \mathbb{E}_{P_t}[\|y^*(t,\xi)\|] < \infty \) for all \( t \in [0,T] \).
Therefore, the pair \( (x^*, y^*) \) forms a classical solution on \( [0,T] \), which completes the proof.
\end{proof}

\begin{remark}\label{remark_unique}
Corollary~\ref{coro:existence} establishes the existence of a classical solution to the DSVI-O \eqref{DSTO_first_stage}–\eqref{DSTO_second_stage}. However, uniqueness is not guaranteed in general, as it typically requires Lipschitz continuity of the right-hand side of the ODE in \eqref{ODE_H}.
Below, we provide a special example in which uniqueness holds. Suppose that for each \( i \in [k] \), the objective function is given by
\begin{equation}\label{strongC}
f_i(t, \xi, \x, \y_i) = g_i(t, \xi, \x, \y_i) + h_i(\y_i),
\end{equation}
where \( g_i \colon [0,T] \times \Xi \times X \times \mathbb{R}^{m_i} \to \mathbb{R} \) is twice continuously differentiable with respect to $\y$, and \( h_i \colon \mathbb{R}^{m_i} \to \mathbb{R} \cup \{+\infty\} \) is proper, l.s.c., and convex (possibly nonsmooth). The feasible set is unconstrained, i.e., \( \mathcal{Y}_i(t,\xi) = \mathbb{R}^{m_i} \) for all \( (t,\xi) \in [0,T] \times \Xi \) and \( i \in [k] \). Moreover, assume that $g_i(t,\xi,\x,\cdot)$ is strongly convex for every $(t,\xi,\x)$, i.e., the Hessian $\nabla^2_{\y_i\y_i} g_i(t,\xi,\x,\y_i)$ is positive definite.

Under these conditions, the unique solution $\hat{y}_i(t,\xi,\x)$ of \eqref{DSTO_second_stage} exists, and the solution mapping $p:=(t,\xi,\x)\mapsto \hat{y}_i(p)$ is Lipschitz continuous by \cite[Corollary 2B.10]{dontchev2009implicit}.
Indeed, for any fixed  $\bar{\y}_i \in \mathcal{S}_i(\bar{p})$, define the set-valued mapping
\[
\Psi(\z) := \nabla_{\y_i} {g}_i(\bar{p},\bar{\y}) + \nabla^2_{\y_i\y_i} {g}_i(\bar{p},\bar{\y})(\z - \bar{\y}) + \partial h_i(\z),
\]
which is the subgradient of the strongly convex function
$$
\psi(\z) := {g}_i(\bar{p},\bar{\y}) + (\z - \bar{\y})^\top\nabla_{\y_i} {g}_i(\bar{p},\bar{\y})
+ \frac{1}{2}(\z - \bar{\y})^\top \nabla^2_{\y_i\y_i} {g}_i(\bar{p},\bar{\y})(\z - \bar{\y}) + h_i(\z).
$$
It is easy to see that $0\in \partial \psi(\bar{\y}_i)=\Psi(\bar{\y}_i).$ The strong convexity of $\psi$ implies that its Fenchel conjugate $\psi^*$ is differentiable with Lipschitz continuous gradient (cf.\ \cite[Proposition 12.60]{rockafellar2009variational}). In particular, $\Psi^{-1} = \nabla \psi^*$ is single-valued and Lipschitz continuous around $0$.
Therefore, the solution map $\mathcal{S}_i$ is Lipschitz continuous by \cite[Corollary 2B.10]{dontchev2009implicit}.

Recall that
$
H(t,\x) = \Pi_X\!\left(\x - \mathbb{E}_{P_t}\big[\Phi\big(t,\xi,\x,\hat{y}(t,\xi,\x)\big)\big]\right) - \x.
$
Since $\Phi(t,\xi,\x,\y)$ is Lipschitz continuous in $\x$ and $\y$, the projection $\Pi_X$ is Lipschitz continuous, and the solution map $\hat{y}(t,\xi,\x)=\mathcal{S}_i(t,\xi,\x)$ is Lipschitz continuous in $\x$, $H$ is Lipschitz continuous in $\x$. Moreover, if $P_t$ is a Feller kernel, $H$ is also continuous in $t$. In particular, the Lipschitz property further implies a linear growth property. Consequently, by the Picard--Lindelöf theorem~\cite[Theorem~3.1]{hale2009ordinary}, the associated differential equation admits a unique $C^1$ solution $x^*(\cdot)$ on the entire interval $[0,T]$. This result partially extends~\cite[Lemma~2.1]{chen2022dynamic} to the unconstrained setting.

\end{remark}

We next establish the feasibility of the solution to the system \eqref{DSTO_first_stage}--\eqref{DSTO_second_stage}. Specifically, we show that if the initial condition belongs to the closed convex set \( X \), then the entire trajectory remains within \( X \) over a finite time horizon.

\begin{lemma}\label{lem:projection_invariance_DI}For any $T>0$,
 if $x_0\in X$, then any weak solution \((x^*,y^*)\) of \eqref{DSTO_first_stage}-\eqref{DSTO_second_stage} satisfies  \(x^*(t) \in X\) for all \(t \in[0,T]\).
\end{lemma}

\begin{proof}
Define the Lyapunov map
$
V(x)=\tfrac12\operatorname{dist}^2(x,X)
        =\tfrac12\|x-\Pi_X(x)\|^{2}.
$
Since $X$ is nonempty, convex and closed, $x\in X$ iff $V(x)=0$ and $V$ is continuously differentiable with $\nabla V(x)=x-\Pi_X(x)$ \cite[Theorem 3.2]{fukushima1992equivalent}.
Since   \((x^*,y^*)\) is a weak solution of \eqref{DSTO_first_stage}-\eqref{DSTO_second_stage}, $x^*(t)$ satisfies
\[
\dot{x}^*(t)\in H(t,x^*(t))
     :=\bigl\{\Pi_X\!\bigl(x^*(t)-\varphi\bigr)-x^*(t):\,\varphi\in F(t,x^*(t))\bigr\},\ {\rm a.e.}\ t\in[0,T],
\]
where $F(t,x^*(t))$ is defined in (\ref{F_t_x}).
By the chain rule \cite[Theorem 2.2]{shevitz1994lyapunov},
for almost every \(t\in[0,T]\)
\begin{equation}\label{dotV}
\frac{d}{dt}V\bigl(x^*(t)\bigr)\;\in\;
\Bigl\{\nabla V\bigl(x^*(t)\bigr)^{\!\top}v:\;
      v\in H\bigl(t,x^*(t)\bigr)\Bigr\}.
\end{equation}
Now fix such a time \(t\) and pick any
\(v\in H\bigl(t,x^*(t)\bigr)\).
Then there exists \(\varphi\in F\bigl(t,x^*(t)\bigr)\) with
\(v=\Pi_X\!\bigl(x^*(t)-\varphi\bigr)-x^*(t)\).
Since $X$ is convex, by the first-order optimality condition of $\min_{\mathbf{x}\in X}\tfrac{1}{2}\|x^*(t)-x\|^2 $, we have
$\bigl\langle x^*(t)-\Pi_X(x^*(t)),\,\Pi_X(x^*(t)-\varphi)-\Pi_X(x^*(t))\bigr\rangle\le0$. Consequently,
\[
\begin{aligned}
&\nabla V\bigl(x^*(t)\bigr)^{\!\top}v
=\bigl\langle x^*(t)-\Pi_X(x^*(t)),\,\Pi_X(x^*(t)-\varphi)-x^*(t)\bigr\rangle \\
&=\bigl\langle x^*(t)-\Pi_X(x^*(t)),\,\Pi_X(x^*(t)-\varphi)-\Pi_X(x^*(t))\bigr\rangle
   -\|x^*(t)-\Pi_X(x^*(t))\|^{2}\\
&\le -\|x^*(t)-\Pi_X(x^*(t))\|^{2}\;\le\;0.
\end{aligned}
\]
Because the right-hand side of (\ref{dotV}) consists entirely of nonpositive values,
we conclude that \(\tfrac{d}{dt}V(x^*(t))\leq 0\) for almost every \(t\in[0,T]\). Since \(V(x_0)=0\) (as \(x_0\in X\)) and \(V\) is non-increasing along the trajectory \(x^*\), it follows that \(V(x^*(t)) \equiv 0\) for all \(t\in[0,T]\). Hence, \(x^*(t)\in X\) for all \(t\in[0,T]\).
\end{proof}

\section{Discrete approximation and convergence analysis}\label{discrete}

In this section, we propose a discrete approximation scheme to compute numerical solutions of   DSVI-O~\eqref{DSTO_first_stage}--\eqref{DSTO_second_stage}. We first apply a time-stepping method based on the forward Euler scheme over a uniform time grid to the ODE in \eqref{DSTO_first_stage}. Next we apply the SAA to compute approximate values of the expectation at each discrete time. Finally, we define piecewise linear continuous-time functions by linear interpolation  using these discrete points and prove the convergence to the $x$ part of a weak solution of  DSVI-O~\eqref{DSTO_first_stage}--\eqref{DSTO_second_stage} as the step size in the time-stepping method goes to zero and the sample size of the SAA goes to infinity.
Throughout Section~3, we assume that the conditions of Theorem~\ref{thm:existence} hold, so that a weak solution exists.

Throughout Section~3, we assume that the conditions of
Theorem~\ref{thm:existence} hold. In addition, we assume that
$H$ is jointly upper semicontinuous on $[0,T]\times X$, and that
$\bar{\kappa}_{\Phi_1}$ and $\bar{\kappa}_{B_i}$, $i\in[k]$, are
uniformly bounded on $[0,T]$.

\subsection{Discrete approximation}\label{sec:time_stepping}

We describe the time-stepping method for the ODE in \eqref{DSTO_first_stage}. For a fixed $T>0$, we define the time step size $h := T/N$ with a fixed $N \in \mathbb{N}$. Consider the uniform partition of the interval $[0,T]$ with nodes $t_\nu := \nu h$ for $\nu = 0,\dots,N$. Given the initial condition $\x_0^N = x(0)=x_0  \in X$, we recursively construct a discrete trajectory $\{\mathbf{x}_\nu^N\}_{\nu=0}^{N} \subset \mathbb{R}^n$ along with measurable selections $\{y_{\nu,i}^N(\xi)\}_{\nu=0}^{N} \subset \mathbb{R}^{m_i}$ for $i\in[k]$ as follows.

{\bf Forward Euler scheme}

\begin{equation} \label{eq:be_scheme}
\left\{\begin{aligned}
\mathbf{x}_{\nu+1}^{N} &= \mathbf{x}_{\nu}^{N}
+ h\left(
\Pi_{X} \left(
\mathbf{x}_{\nu}^{N}
- \mathbb{E}_{P_{t_\nu}} \left[
\Phi(t_\nu, \xi, \mathbf{x}_\nu^N, y_\nu^N(\xi))
\right]
\right)
- \mathbf{x}_{\nu}^{N} \right), \\[4pt]
y_{\nu,i}^N(\xi) &\in \arg\min_{\mathbf{y}_i \in \mathcal{Y}_i(t_\nu, \xi)}
f_i(t_\nu, \xi, \mathbf{x}_\nu^N, \mathbf{y}_i), \quad i \in [k],
\end{aligned}\right.
\end{equation}
where
$
y_\nu^N(\xi) := \left(y_{\nu,1}^N(\xi), \dots, y_{\nu,k}^N(\xi)\right).
$

The forward Euler scheme \eqref{eq:be_scheme} involves an intractable expectation with respect to the distribution \( P_{t_\nu} \) at each time \( t_\nu \). To approximate this expectation, we apply the SAA method.

At each node \( t_\nu \), we draw a batch of i.i.d.\ samples \( \{\xi_\nu^{j}\}_{j=1}^{J} \) from \( P_{t_\nu} \), where the batch size \( J \in \mathbb{N} \) is independent of \( \nu \). Replacing the expectation by the empirical mean leads to the following SAA--Euler scheme:

{\bf SAA-Euler scheme}
\begin{equation} \label{eq:be_saa_compact}
\left\{\begin{aligned}
\x_{\nu+1}^{N,J}
&=
\x_{\nu}^{N,J}
+\,h\Big(
        \Pi_{X}\big(
            \x_{\nu}^{N,J}
            -\tfrac{1}{J}\sum_{j=1}^{J}
               \Phi\big(t_{\nu},\xi_{\nu}^{j},\x_{\nu}^{N,J},
                         y_{\nu}^{N,J}(\xi_{\nu}^{j})\big)
        \big)
        -\x_{\nu}^{N,J}
     \Big),\\[4pt]
y_{\nu,i}^{N,J}\big(\xi_{\nu}^{j}\big)
&\in
\arg\min_{\mathbf{y}_i\in\mathcal{Y}_i(t_{\nu},\xi_{\nu}^{j})}
      f_i\big(t_{\nu},\xi_{\nu}^{j},\x_{\nu}^{N,J},\mathbf{y}_i\big),
\quad i\in[k],\; j\in[J],
\end{aligned}\right.
\end{equation}
where
$
y_{\nu}^{N,J}(\xi_{\nu}^{j})
:= \big(
    y_{\nu,1}^{N,J}(\xi_{\nu}^{j}),\dots,
    y_{\nu,k}^{N,J}(\xi_{\nu}^{j})
\big).
$

We define the piecewise linear interpolations $x^N, x^{N,J}  : [0,T] \to \mathbb{R}^n$, respectively, as follows
\begin{equation} \label{eq:interp_x}
x^N(t) = \frac{t_{\nu+1} - t}{h}\, \mathbf{x}_\nu^N + \frac{t - t_\nu}{h}\, \mathbf{x}_{\nu+1}^N,
\quad t \in [t_\nu, t_{\nu+1}],\ \nu = 0,\dots,N-1.
\end{equation}
\begin{equation} \label{eq:interp_y}
x^{N,J}(t) = \frac{t_{\nu+1} - t}{h}\, \x_\nu^{N,J} + \frac{t - t_\nu}{h}\, \x_{\nu+1}^{N,J},
\quad t \in [t_\nu, t_{\nu+1}],\ \nu = 0,\dots,N-1.
\end{equation}

\subsection{Convergence analysis}\label{subsec:conv} By the proof of Theorem~\ref{thm:existence}, we know that the set-valued map $H(t,\mathbf{x}) = \{\Pi_{X}(\mathbf{x} - \varphi) - \mathbf{x} : \varphi \in F(t,\mathbf{x})\}$ with $F(t,\mathbf{x})$ defined in (\ref{F_t_x}) satisfies that $H(\cdot,\mathbf{x})$ is measurable and $H(t,\cdot)$ is u.s.c. However, $H$ may fail to be jointly u.s.c. in $(t, \mathbf{x})$, which prevents a direct application of classical convergence results for difference inclusions~\cite{dontchev1992difference}. In some special cases, for instance when $P_t$ is Feller together with suitable uniform integrability conditions, the map $H(t,\mathbf{x})$ can be jointly u.s.c. in $(t,\mathbf{x})$. To address the lack of joint u.s.c.\ only under Assumptions \ref{A1}–\ref{A4}, we construct an auxiliary set-valued map $H_\epsilon$ that is jointly u.s.c.\ and satisfies the same linear growth condition. Moreover, $H_\epsilon(t,\mathbf{x}) \subset H(t,\mathbf{x})$, and both inclusions admit the same set of solutions on a large subset of $[0,T]$. The key properties of this approximating inclusion are summarized in the following lemma.

\subsection{Convergence analysis}\label{subsec:conv}
By the proof of Theorem~\ref{thm:existence}, we know that the set-valued map
$H(t,\mathbf{x}) = \{\Pi_{X}(\mathbf{x} - \varphi) - \mathbf{x} : \varphi \in F(t,\mathbf{x})\}$
with $F(t,\mathbf{x})$ defined in (\ref{F_t_x}) satisfies that $H(\cdot,\mathbf{x})$ is measurable and $H(t,\cdot)$ is u.s.c. However, $H$ may fail to be jointly u.s.c. in $(t, \mathbf{x})$, which prevents a direct application of classical convergence results for difference inclusions~\cite{dontchev1992difference}. 
In some special cases, for instance when $P_t$ is Feller
together with suitable uniform integrability conditions,
the map $H(t,\mathbf{x})$ can be jointly u.s.c. in $(t,\mathbf{x})$. 
To address the lack of joint u.s.c., we construct an auxiliary set-valued map $H_\epsilon$ that is jointly u.s.c.\ and satisfies the same linear growth condition. Moreover, $H(t,\mathbf{x})\subseteq H_\epsilon(t,\mathbf{x})$,
and the two maps coincide outside a finite set of times.
Consequently, they define the same differential inclusion
a.e. on $[0,T]$.
The key properties of this approximating inclusion are summarized in the following lemma.

\begin{lemma}\label{lem:H_eps}
Suppose that there exists a finite set
$E\subset[0,T]$ such that $H$ is jointly u.s.c. on
$([0,T]\setminus E)\times X$, and that
$\bar{\kappa}_{\Phi_1}$ and $\bar{\kappa}_{B_i}$,
$i\in[k]$, are uniformly bounded on $[0,T]$.
For any \( \epsilon > 0 \), there exists a closed set \(T_\epsilon \subset [0,T] \) satisfying \(\lambda([0,T] \setminus T_\epsilon) \le \epsilon\) with $\lambda(\cdot)$ being the  Lebesgue measure on $[0,T]$, and a jointly u.s.c. set-valued map $H_\epsilon : [0,T] \times X \rightrightarrows \mathbb{R}^n$ with closed convex values such that $H_\epsilon$ satisfies the linear growth condition and any sequence $\{x^N\}_{N \in \mathbb{N}}$ generated by the time-stepping scheme \eqref{eq:be_scheme} and \eqref{eq:interp_x} satisfies, for all sufficiently large $N$,
$
\dot{x}^N(t) \in H_\epsilon(t_\nu, \bar{x}^N(t))
$ for a.e. $t \in [0,T]$ with $\bar{x}^N(t) := \mathbf{x}^N_\nu$ for $t \in (t_\nu, t_{\nu+1}),$ $\nu=0,\dots,N-1.$
\end{lemma}

\begin{proof}
By \eqref{linear_growth} and the uniform boundedness
assumptions, there exist constants $a,b>0$ such that
$
\sup_{\v\in H(t,\x)}\|\v\|\le a\|\x\|+b
\quad\text{on }[0,T]\times X.
$
For every $\epsilon>0$, choose a closed set
$T_\epsilon\subset[0,T]\setminus E$ satisfying
$\lambda([0,T]\setminus T_\epsilon)\le\epsilon$.
Define
$H_\epsilon(t,\x):=H(t,\x)$ if $t\notin E$, and
$H_\epsilon(t,\x):=\overline{\mathbb B}(0,a\|\x\|+b)$ if $t\in E$.
Then $H_\epsilon$ has nonempty compact convex values,
contains $H$, and satisfies the same linear-growth bound.
Moreover, $H_\epsilon=H$ on
$([0,T]\setminus E)\times X$, and hence on
$T_\epsilon\times X$.
Since $E$ is finite, $H_\epsilon$ is jointly u.s.c. at
every $(t,\x)$ with $t\notin E$.
At a point $(t,\x)$ with $t\in E$, upper semicontinuity
follows from
$
H_\epsilon(s,\z)
\subseteq\overline{\mathbb B}(0,a\|\z\|+b)
$
and the continuity of $\z\mapsto a\|\z\|+b$.
Thus $H_\epsilon$ is jointly u.s.c. on $[0,T]\times X$.
When $0<h\le1$, the convexity of $X$ and
\eqref{eq:be_scheme} imply that $\x_\nu^N\in X$.
Therefore,
\[
\dot x^N(t)
\in H(t_\nu,\bar x^N(t))
\subseteq H_\epsilon(t_\nu,\bar x^N(t))
\]
for a.e. $t\in(t_\nu,t_{\nu+1})$.
This completes the proof.
\end{proof}

The following theorem  shows the uniform convergence of  $\{x^N\}_{N\in\N}$.

\begin{theorem}\label{thm:euler-conv}
Let $\{x^{N}\}_{N\in\mathbb{N}}$ be the sequence generated by \eqref{eq:interp_x}. Then, under Assumptions~\ref{A1}--\ref{A4} and the assumption of Lemma~\ref{lem:H_eps}, every sequence $\{x^N\}_{N \in \mathbb{N}}$ has a uniformly convergent subsequence $\{x^{N_\tau}\}_{\tau \in \mathbb{N}}$ and a limit $x$ such that
$$
\sup_{t \in [0,T]} \|x^{N_\tau}(t) - x(t)\| \to 0 \quad \text{as } \tau \to \infty,
$$
and the limit $x$ is the $x$-part of a weak solution of \eqref{DSTO_first_stage}-\eqref{DSTO_second_stage}.

\end{theorem}

\begin{proof}
Let $\epsilon_{i}\downarrow0$ as $i\to \infty$. For each \(i\), by Lemma \ref{lem:H_eps} and \cite[Theorem 2.2]{dontchev1992difference},  we have a closed set \(T_{\epsilon_{i}}\subset[0,T]\), a
jointly u.s.c. set-valued map \(H_{\epsilon_{i}}\), and a subsequence $\{{x^{N}(\cdot)\}}_{N\in \N_{\epsilon_i}\subset \N}$
of $\{x^N(\cdot)\}_{N\in\N}$
that converges uniformly on \([0,T]\)
to an absolutely continuous limit
$
x_{i}:[0,T]\to\mathbb R^{n}$ satisfies that $
\dot x_{i}(t)\in H_{\epsilon_{i}}\bigl(t,x_{i}(t)\bigr)$ for a.e. $t\in[0,T]
$ and, since $H_{\epsilon_i}=H$ outside the finite set $E$,
$\dot x_i(t)\in H(t,x_i(t))$
for a.e. $t\in[0,T]$.
Moreover, $x_i(0)=x_0$ and $x_i(t)\in X$ for all $t\in[0,T]$.

In addition, by \cite[Theorem.~7.1]{deimling2011multivalued}, we can extract a uniformly convergent subsequence $\{x_{{i_\tau}}(\cdot)\}_{\tau\in\N}$ converging to an absolutely continuous limit $x:[0,T]\to\mathbb{R}^n$ satisfying $\dot{x}(t)\in H(t,x(t))$ a.e. on $[0,T]$. For each $\tau$, choose  $K(\tau)\in \N_{\epsilon_{i_\tau}} $ large enough to satisfy
$
\sup_{t\in[0,T]}\|x^{N}(t)-x_{{i_\tau}}(t)\|<\frac{1}{\tau}
$ for any $N\in \N_{\epsilon_{i_\tau}}\bigcap [K(\tau),\infty)$. In particular, choose $K(\tau)$ to be strictly increasing in $\tau \in \mathbb{N}$.
Define $N_\tau:=K(\tau)$. The subsequence $\{x^{N_\tau}\}_{\tau\in\N}$ satisfies
\[
\sup_{t\in[0,T]}\|x^{N_\tau}(t)-x(t)\|\le\frac{1}{\tau}+\sup_{t\in[0,T]}\|x_{{i_\tau}}(t)-x(t)\|\to 0\quad \text{as } \tau \to \infty.
\]
Finally, by the measurable selection argument in the proof of
Theorem~\ref{thm:existence}, $x$ is the $x$-part
of a weak solution of
\eqref{DSTO_first_stage}--\eqref{DSTO_second_stage}.
This completes the proof.
\end{proof}

\begin{remark}
A notable property of the coupled time-stepping scheme (\ref{eq:be_scheme}) is that the sequence of first-stage trajectories $\{x^N\}_{N\in\N}$ admits a uniformly convergent subsequence for any selections $\{y^N\}_{N\in\N}$. This indicates that the first-stage dynamics can remain stable with respect to fluctuations in the second-stage variables. To illustrate this, we give the following example. Let $\xi \sim \mathrm{U}(0,1)$, $P$ denote the corresponding probability distribution on $\Xi=[0,1]$,
$X=\mathbb{R}$, $k=1$, $\mathcal Y(t,\xi)=\{\y\in\bbr^2_+:\;\y_1+\y_2=1\}$ and
\[
f(t,\xi,\x,\y) = (1+\x^2+\xi)\,(\y_1+\y_2),\quad
\Phi(t,\xi,\x,\y) = \x + \sin(2\pi\xi) + (\y_1 - \y_2).
\]
Then the optimal set satisfies
$\mathcal S(t,\xi,\x) = \mathcal Y(t,\xi)$ for all $(t,\xi,\x)$.

Define a function $r_\nu(\xi):=(-1)^{\lfloor2^{\nu+1}\xi\rfloor}$ and let
$
y_\nu^N(\xi) = (\tfrac{1 + r_\nu(\xi)}{2}, \tfrac{1 - r_\nu(\xi)}{2})$
for each $\nu \in \mathbb{N}$.
Then $y_\nu^N(\xi)\in \mathcal S(t_\nu, \xi, \x_\nu^N)$, $\mathbb{E}_P[y_{\nu,1}^N(\xi) - y_{\nu,2}^N(\xi)] = \mathbb{E}_P[r_\nu] = 0$ and $\mathbb{E}_P[\sin(2\pi\xi)] = 0$.  Thus the Euler update reduces to
$$
\x_{\nu+1}^N = \x_\nu^N + h(\x_\nu^N - \mathbb{E}_P[\Phi] - \x_\nu^N) = (1 - h)\,\x_\nu^N,
$$
which implies $\x_\nu^N = (1 - h)^\nu \x_0$. After interpolation, the first-stage trajectories converge uniformly to $x(t) = x_0 e^{-t}$. However, define
$
y^{N}(t,\xi)
:= y_\nu^N(\xi),\ t \in [t_\nu, t_{\nu+1})$. For $M\ge2N$, we have
$
\|y^N-y^M\|_{L^2([0,T]\times[0,1])}^2
=T-\frac{T}{M}\ge\frac{T}{2}.
$
Thus no subsequence is strongly convergent in this space.
This example illustrates that the sequence $\{x^N\}$ may still admit a uniformly convergent subsequence, even when the second-stage selections $\{y^N\}$ do not admit any convergent subsequence.

 Moreover,  this example also shows that the selection rule plays a critical role in the joint convergence of $(x^N, y^N)$. For instance, employing the constant selector $y(t,\xi) \equiv (1,0)$ yields convergent subsequences for both stages.
 Note that the function $f$ is continuously differentiable and convex with respect to $\y$, and thus the minimization problem $\min_{\y\in {\cal Y}(t,\xi)}f(t,\xi,\x,\y)$ is equivalent to a variational inequality problem. For continuous selection of solutions
to parametric variational inequalities, see \cite{ChenXiang2011,Han-JSPang2024}.
\end{remark}

Theorem~\ref{thm:euler-conv} establishes the subsequential uniform convergence of the
continuous-time interpolants $\{x^{N}\}$ generated by the Euler scheme, and shows that any
uniform limit is the $x$-part of a weak solution to
\eqref{DSTO_first_stage}--\eqref{DSTO_second_stage}.
We next strengthen this result by deriving an error bound analysis on $[0,T]$.

Throughout the analysis, we additionally assume the Riemann integrability of
$\bar{\kappa}_{\Phi_1}(t)$ and $\bar{\kappa}_{B_i}(t)$, $i\in[k]$.
Set
$
M_i:=\sup_{t\in[0,T]}\widetilde L_i(t),\ i=1,2,\
R:=e^{TM_1}\big(\|x_0\|+TM_2\big),$ and $
D:=X\cap\overline{\mathbb{B}}(0,R),
$
where $\widetilde L_1(\cdot)$ and $\widetilde L_2(\cdot)$ are given by \eqref{linear_growth}.
Then any solution $x(\cdot)$ of \eqref{ODE1} on $[0,T]$ satisfies $\|x(t)\|\le R$ for all
$t\in[0,T]$ by Gr\"onwall's inequality. Moreover, we assume the support set $\Xi$ is compact. We then give the following assumption.

\begin{assumption}\label{A_rate}
\begin{enumerate}[leftmargin=3em]
\item[\textup{(a)}]
For each $i\in[k]$ and for $\xi\in \Xi$, there exists a measurable function $\alpha_i:\Xi\to(0,\infty)$ and a constant $\bm{\alpha}_0>0$ such that
$\alpha_i(\xi)\ge \bm{\alpha}_0$ for all $\xi\in\Xi$, and for all $t\in[0,T]$, $\x\in X$, and $\y\in \overline{\mathcal{Y}}_i$,
\[
f_i(t,\xi,\x,\y)-v_i(t,\xi,\x)\ge\alpha_i(\xi)\,\dist\!\big(\y,\mathcal S_i(t,\xi,\x)\big),
\]
with $v_i(t,\xi,\x):=\min_{\y\in\mathcal Y_i(t,\xi)} f_i(t,\xi,\x,\y).$

\item[\textup{(b)}]
For each $i\in[k]$ and for $\xi\in \Xi$, the map $f_i(\cdot,\xi,\cdot,\cdot)$ is Lipschitz continuous in $(t,\x,\y)$ on $[0,T]\times X\times\overline{\mathcal Y}_i$,
$\mathcal Y_i(\cdot,\xi)$ is Hausdorff-Lipschitz continuous on $[0,T]$, and
$\Phi_1(\cdot,\xi,\cdot)$ and $B_i(\cdot,\xi,\cdot)$ are Lipschitz continuous in $(t,\x)$ on $[0,T]\times X$,
with measurable and bounded moduli $\kappa_{f_i}(\xi)$, $\kappa_{\mathcal Y_i}(\xi)$, $\tilde{\kappa}_{\Phi_1}(\xi)$, and $\tilde{\kappa}_{B_i}(\xi)$, respectively.

\item[\textup{(c)}]
There exists $L_P>0$ such that for all $t,t'\in[0,T]$,
\[
\|P_t-P_{t'}\|_{\mathrm{TV}}\le L_P\,|t-t'|,
\]
where $\|\cdot\|_{\mathrm{TV}}$ denotes the total variation distance; see \cite[Definition 2.4]{Tsybakov2009}.
\end{enumerate}
\end{assumption}

\begin{remark}
Under the standing assumptions of this section, the
existence of a weak solution follows from
Theorem~\ref{thm:existence}.
Moreover, Assumption~\ref{A_rate}\textup{(a)} is the weak sharp minima
property in the sense of Burke--Ferris \cite{BurkeFerris1993}. It allows
$\mathcal S_i(t,\xi,\x)$ to be multi-valued, and together with the remaining
parts of Assumption~\ref{A_rate}, it yields the Hausdorff-Lipschitz
continuity of the map $(t,\x)\mapsto H(t,\x)$.

We further illustrate that Assumption~\ref{A_rate}\textup{(c)} holds for several
standard classes of probability measures by three examples.
\begin{enumerate}[label=\textup{(\roman*)}, leftmargin=0pt, itemindent=2em]
\item Let $\Xi=\{z_1,\ldots,z_m\}$ and
$
P_t=\sum_{j=1}^m p_j(t)\delta_{z_j}.
$
If each $p_j$ is a Lipschitz continuous function on $[0,T]$ with a constant $L_j>0$, then
\[
\|P_t-P_{t'}\|_{\mathrm{TV}}
=
\frac12\sum_{j=1}^m |p_j(t)-p_j(t')|
\le
\frac12\sum_{j=1}^m L_j |t-t'|.
\]
Hence Assumption~\ref{A_rate}\textup{(c)} holds with
$
L_P=\frac12\sum_{j=1}^m L_j.
$

\item Let
$
P_t=\mathrm{Unif}([a(t),b(t)]^\ell)
$
on $\mathbb R^\ell$ with
$
a(t)=a_0-tc_0,\  b(t)=b_0+tc_0,
$
where $a_0<b_0$ and $c_0>0$. Set
$
I_t=[a(t),b(t)]^\ell
$
and
$
r(t)=b(t)-a(t)=b_0-a_0+2tc_0.
$
For $t\ge t'$, one has $I_{t'}\subset I_t$. Since the density of $P_t$ is
$
r(t)^{-\ell}\mathbf 1_{I_t},
$
by \cite[Lemma 2.1]{Tsybakov2009}, we obtain
\[
\begin{aligned}
\|P_t-P_{t'}\|_{\mathrm{TV}}
&=
\frac12\int_{\mathbb R^\ell}
\left|
\frac{1}{r(t)^\ell}\mathbf 1_{I_t}(z)
-
\frac{1}{r(t')^\ell}\mathbf 1_{I_{t'}}(z)
\right|dz                          \\
&
\le
\frac{2\ell c_0(t-t')}{b_0-a_0+2tc_0}
\le
\frac{2\ell c_0}{b_0-a_0}|t-t'|.
\end{aligned}
\]
The case $t<t'$ is analogous. Thus Assumption~\ref{A_rate}\textup{(c)} holds
with
$
L_P=2\ell c_0/(b_0-a_0).
$

  \item Let
$
P_t=\mathcal N(m(t),\Sigma)
$
on $\mathbb R^\ell$, where $\Sigma$ is a symmetric positive definite matrix and $m$ is a
Lipschitz continuous function with a constant $L_m>0$. By Pinsker's inequality
\cite[Lemma~2.5]{Tsybakov2009} and the
Kullback-Leibler divergence between Gaussian distributions with the same
covariance matrix \cite[Chapter~9, Section 5]{CoverThomas2006},
\[
\begin{aligned}
\|P_t-P_{t'}\|_{\mathrm{TV}}
&\le
\frac12
\left(
(m(t)-m(t'))^\top
\Sigma^{-1}
(m(t)-m(t'))
\right)^{1/2}                            \\
&      \le
\frac{L_m}{2\sqrt{\lambda_{\min}(\Sigma)}}|t-t'|.
\end{aligned}
\]
Therefore Assumption~\ref{A_rate}\textup{(c)} holds with
$
L_P=L_m/(2\sqrt{\lambda_{\min}(\Sigma)}).
$
\end{enumerate}
\end{remark}

\begin{lemma}[Hausdorff-Lipschitz continuity of $H$]\label{lem:H_t_x_HL}
Under Assumption~\ref{A_rate}, there exists a constant $L_H>0$ such that for all $t,t'\in[0,T]$ and all $\x,\x'\in D$,
\[
\HH\big(H(t,\x),H(t',\x')\big)\ \le\ L_H\,\big(|t-t'|+\|\x-\x'\|\big).
\]
\end{lemma}

\begin{proof}
Fix $t,t'\in[0,T]$ and $\x,\x'\in D$, and set $\Delta:=|t-t'|+\|\x-\x'\|$.
Fix $i\in[k]$ and $\xi\in\Xi$. Take any $\y'\in\mathcal S_i(t',\xi,\x')\subseteq\mathcal Y_i(t',\xi)$.
By Assumption~\ref{A_rate}\textup{(b)},
\[
f_i(t,\xi,\x,\y')\le f_i(t',\xi,\x',\y')+\kappa_{f_i}(\xi)\Delta
= v_i(t',\xi,\x')+\kappa_{f_i}(\xi)\Delta.
\]
Let $\y\in \mathcal{S}_i(t,\xi,\x)$. Since $\mathcal Y_i(\cdot,\xi)$ is Hausdorff-Lipschitz with modulus
$\kappa_{\mathcal Y_i}(\xi)$, there exists $\bar\y\in\mathcal Y_i(t',\xi)$ such that
$\|\bar\y-\y\|\le \kappa_{\mathcal Y_i}(\xi)|t-t'|$.
Using Assumption~\ref{A_rate}\textup{(b)} again,
\[
\begin{aligned}
v_i(t',\xi,\x')
\le f_i(t',\xi,\x',\bar\y)
&\le f_i(t,\xi,\x,\y)+\kappa_{f_i}(\xi)\big(|t-t'|+\|\x-\x'\|+\|\bar\y-\y\|\big)\\
&\le v_i(t,\xi,\x)+(1+\kappa_{\mathcal Y_i}(\xi))\kappa_{f_i}(\xi)\Delta.
\end{aligned}
\]
Combining the above inequalities yields
$
f_i(t,\xi,\x,\y')\le v_i(t,\xi,\x)+(2+\kappa_{\mathcal Y_i}(\xi))\kappa_{f_i}(\xi)\Delta,
$
and Assumption~\ref{A_rate}\textup{(a)} implies
$
\dist\!\big(\y',\mathcal S_i(t,\xi,\x)\big)
\le \frac{(2+\kappa_{\mathcal Y_i}(\xi))\,\kappa_{f_i}(\xi)}{\bm{\alpha}_0}\,\Delta.
$
Swapping $(t,\x)$ and $(t',\x')$ gives
\begin{equation}\label{eq:H_S_HL}
\HH\big(\mathcal S_i(t,\xi,\x),\mathcal S_i(t',\xi,\x')\big)
\le \frac{(2+\kappa_{\mathcal Y_i}(\xi))\,\kappa_{f_i}(\xi)}{\bm{\alpha}_0}\,\Delta.
\end{equation}

Recalling $\mathcal G$ in (\ref{G_fun}), using translation invariance and subadditivity of $\HH$, we have
\begin{align}
\HH\big(\mathcal G(t,\xi,\x),\mathcal G(t',\xi,\x')\big)
\le&\ \|\Phi_1(t,\xi,\x)-\Phi_1(t',\xi,\x')\|\nonumber\\
&+\sum_{i=1}^k
\HH\big(B_i(t,\xi,\x)\mathcal S_i(t,\xi,\x),\,B_i(t',\xi,\x')\mathcal S_i(t',\xi,\x')\big).\label{eq:hau_G}
\end{align}
Fix $i\in[k]$. By Assumption~\ref{A_rate}\textup{(b)} and the compactness of $\Xi$, it follows that for any $t\in[0,T]$, $\x\in D$, and $\xi\in\Xi$
\[
\|B_i(t,\xi,\x)\|
\le \|B_i(0,\xi,x_0)\|+\tilde{\kappa}_{B_i}(\xi)\big(T+2R\big)<\infty.
\]
Moreover, $\mathcal S_i(t,\xi,\x)\subseteq \overline{\mathcal Y}_i$ implies
the boundedness of $\mathcal S_i(t,\xi,\x)$.
Combining \eqref{eq:H_S_HL} with the Lipschitz continuity of
$\Phi_1(\cdot,\xi,\cdot)$ and $B_i(\cdot,\xi,\cdot)$ in Assumption~\ref{A_rate}\textup{(b)} and using \eqref{eq:hau_G},
we obtain
\[
\HH\big(\mathcal G(t,\xi,\x),\mathcal G(t',\xi,\x')\big)\le C_{\mathcal G}(\xi)\,\Delta,
\]
where
$
C_{\mathcal G}(\xi):=\tilde{\kappa}_{\Phi_1}(\xi)
+\sum_{i=1}^k\Big(
\big(\|B_i(0,\xi,x_0)\|+\tilde{\kappa}_{B_i}(\xi)(T+2R)\big)
\frac{(2+\kappa_{\mathcal Y_i}(\xi))\,\kappa_{f_i}(\xi)}{\bm{\alpha}_0}
+\tilde{\kappa}_{B_i}(\xi)\,M
\Big)$ with $M :=\max_{i\in[k]}\;\sup_{\mathbf{y} \in \overline{\mathcal{Y}}_{i}} \, \|\mathbf{y}\|$.
Since the moduli in Assumption~\ref{A_rate}\textup{(b)} are bounded on $\Xi$, there exists a constant $L_{\mathcal G}>0$ such that
$C_{\mathcal G}(\xi)\le L_{\mathcal G}$ for all $\xi\in\Xi$, and hence
\[
\HH\big(\mathcal G(t,\xi,\x),\mathcal G(t',\xi,\x')\big)\le L_{\mathcal G}\,\Delta.
\]

Next, by the definition of the Aumann integral under the fixed measure $P_t$, we have
\begin{equation}\label{eq:EG_same_measure}
\HH\!\left(\E_{P_t}[\mathcal G(t,\xi,\x)],\,\E_{P_t}[\mathcal G(t',\xi,\x')]\right)
\le \E_{P_t}\!\big[\HH(\mathcal G(t,\xi,\x),\mathcal G(t',\xi,\x'))\big]
\le L_{\mathcal G}\,\Delta.
\end{equation}
For the fixed $\x'\in D$ and $t'\in[0,T]$ and for any  $\|u\|\le 1$, define
$
h_{\mathcal G(t',\xi,\x')}(u):=\sup_{g\in\mathcal G(t',\xi,\x')}u^\top g.
$
By (\ref{G_fun}) and Assumption \ref{A_rate}(b) and the compactness of $\Xi$,
there exists a constant $\overline{\Gamma}_R>0$ such that
$\|h_{\mathcal G(t',\cdot,\x')}(u)\|_\infty\le \overline{\Gamma}_R$ for all $\|u\|\le 1$.
Using $h_{\E_{P}[\mathcal G]}(u)=\E_{P}[h_{\mathcal G}(u)]$ and \cite[Proposition D.2.4]{DM2018MarkovChains}, we obtain
\begin{align}
&\HH\!\left(\E_{P_t}[\mathcal G(t',\xi,\x')],\,\E_{P_{t'}}[\mathcal G(t',\xi,\x')]\right)\nonumber\\
=&\sup_{\|u\|\le 1}\Big|\E_{P_t}[h_{\mathcal G(t',\xi,\x')}(u)]-\E_{P_{t'}}[h_{\mathcal G(t',\xi,\x')}(u)]\Big|\nonumber\\
\le& \sup_{\|u\|\le 1} \|h_{\mathcal G(t',\cdot,\x')}(u)\|_\infty\,\|P_t-P_{t'}\|_{\mathrm{TV}}
\le \overline{\Gamma}_R\,\|P_t-P_{t'}\|_{\mathrm{TV}}\nonumber\\
\le& \overline{\Gamma}_R\,L_P\,|t-t'|
\le \overline{\Gamma}_R\,L_P\,\Delta. \label{eq:EG_change_measure}
\end{align}

Recalling $F(t,\x)=\E_{P_t}[\mathcal G(t,\xi,\x)]$, by \eqref{eq:EG_same_measure} and \eqref{eq:EG_change_measure}, we have
\[
\HH\big(F(t,\x),F(t',\x')\big)
\le (L_{\mathcal G}+\overline{\Gamma}_R L_P)\,\Delta.
\]
Finally, by nonexpansiveness of $\Pi_X$ and the definition of $H$, it follows
\[
\HH\big(H(t,\x),H(t',\x')\big)
\le 2\|\x-\x'\|+\HH\big(F(t,\x),F(t',\x')\big)
\le \big(2+L_{\mathcal G}+\overline{\Gamma}_R L_P\big)\Delta.
\]
Thus, the proof is completed with $L_H:=2+L_{\mathcal G}+\overline{\Gamma}_R L_P$.
\end{proof}

\begin{theorem}[Euler discretization error]\label{thm:euler_rate_compact}
Under Assumption~\ref{A_rate}, for any $x$-part $x(\cdot)$ of a weak solution to (\ref{DSTO_first_stage})-(\ref{DSTO_second_stage}) on $[0,T]$, there exists an Euler sequence
$\{\mathbf x_\nu^{N}\}_{\nu=0}^{N}$ such that
\[
\max_{0\le \nu\le N}\big\|x(t_\nu)-\mathbf x_\nu^{N}\big\|=O(h).
\]
\end{theorem}
\begin{proof}
By Lemma~\ref{lem:H_t_x_HL} and \eqref{linear_growth}, the set-valued map $(t,\x) \mapsto H(t,\x)$ is Hausdorff-Lipschitz continuous and bounded on $[0,T]\times D$.
Hence, applying \cite[Proposition, p.~350]{DontchevFarkhi1989} completes the proof.
\end{proof}

Using Theorem~\ref{thm:euler_rate_compact},
the next corollary shows that, along an Euler sequence $\{\mathbf x_\nu^N\}$ generated by \eqref{eq:be_scheme},
there exists a selection of second-stage solutions $\{y_{\nu,i}^N\}$ that inherits the convergence of $\{\mathbf x_\nu^N\}$.

\begin{corollary}\label{thm:y_rate}
Fix $i\in[k]$ and $\xi\in\Xi$.
Define, for each $t\in[0,T]$,
$
y_i(t,\xi):=\Pi_{\mathcal S_i(t,\xi,x(t))}(0),
$
and for each
$t_\nu$,
$
y_{i,\nu}^N(\xi):=\Pi_{\mathcal S_i(t_\nu,\xi,\mathbf x_\nu^N)}\big(y_i(t_\nu,\xi)\big),\ \nu=0,\ldots,N.
$
Then, under Assumption~\ref{A_rate}, we have
\begin{equation}\label{eq:y_Oh_thm}
\max_{0\le \nu\le N}\big\|y_{i,\nu}^N(\xi)-y_i(t_\nu,\xi)\big\|
=O(h).
\end{equation}
\end{corollary}

\begin{proof}
For $\nu=0,\dots,N$, by the definition of $y_{i,\nu}^N(\xi)$ and $y_i(t_\nu,\xi)\in \mathcal S_i(t_\nu,\xi,x(t_\nu))$, we have
\[
\begin{aligned}
\big\|y_{i,\nu}^N(\xi)-y_i(t_\nu,\xi)\big\|
&=\dist\!\big(y_i(t_\nu,\xi),\mathcal S_i(t_\nu,\xi,\mathbf x_\nu^N)\big)\\
&
\le \HH\!\big(\mathcal S_i(t_\nu,\xi,x(t_\nu)),\,\mathcal S_i(t_\nu,\xi,\mathbf x_\nu^N)\big).
\end{aligned}
\]
Therefore, using \eqref{eq:H_S_HL} and noting the compactness of $\Xi$, the result \eqref{eq:y_Oh_thm} follows from Theorem~\ref{thm:euler_rate_compact}.
\end{proof}

The SAA-Euler scheme (\ref{eq:be_saa_compact}) replaces the true expectations in the forward Euler method with empirical averages and provides a fully implementable version. We now give a convergence result for the SAA-Euler scheme (\ref{eq:be_saa_compact}).
\begin{theorem}\label{thm:saa-subseq}
{Under Assumptions\ \ref{A1}--\ref{A4}}, every sequence $\{x^{N,J}\}_{ N,J\in \N}$ defined in \eqref{eq:be_saa_compact}  has a uniformly convergent subsequence $\{x^{N_j,J_j}\}_{j\in\N}$ in $[0,T]$ with $N_j\to\infty$ and $J_j\to\infty$ as $j\to\infty$ to an absolutely continuous function $\bar{x}$ such that for any $r>0$,
\begin{equation}\label{SAA_DI}
  \dot{\bar{x}}(t)\in {\operatorname{conv}}(H^r(t,\bar{x}(t))),\ \mathrm{a.e.}\ t\in[0,T],
\end{equation}
where $H^r(t,\bar{x}(t)):= \bigcup_{\x' \in \B(\bar{x}(t),r)} H(t,\x'),$
$\mathbb{B}(\bar{x}(t), r)$ denotes the open ball of radius $r$ centered at $\bar{x}(t)$ in $\mathbb{R}^n$.
\end{theorem}

\begin{proof}
Recall that \eqref{DSTO_first_stage}-\eqref{DSTO_second_stage} can be rewritten as the differential inclusion (\ref{ODE1}) and the forward Euler method \eqref{eq:be_scheme} can be written as
$
\x_{\nu+1}^N \in \x_{\nu}^N + h H\big(t_\nu,\x_\nu^N\big),
$
for $\nu=0,1,\ldots,N-1$ with $\x_0^N=x_0$.  Similarly, the SAA--Euler scheme \eqref{eq:be_saa_compact} can be expressed as
$
\x_{\nu+1}^{N,J} \in \x_{\nu}^{N,J} +hH_J\big(t_\nu,\x_\nu^{N,J}\big),
$
where
$
H_J(t,\x) := \big\{\Pi_X(\mathbf{x} - \varphi - \mathbf{x} : \varphi\in \frac{1}{J}\sum_{j=1}^{J}\mathcal{G}(t,\xi^{j},\x)\big\}
$
with $\{\xi^j\}_{j=1}^J$ being i.i.d. samples drawn from the distribution $P_{t_\nu}$ at each time node $t_\nu$. In the following, we consider the case where $F$ is jointly u.s.c. in $(t,x)$. When $F(\cdot,\x)$ is measurable and $F(t,\cdot)$ is u.s.c., similar results follow from arguments as in Lemma~\ref{lem:H_eps}.

Now we show that the SAA field \(H_J\) inherits the linear–growth
property of \(H\).
Fix $r>0$,
by \cite[Theorem 2]{shapiro2007uniform}, it follows
\begin{equation}\label{eq:ShapiroHaus}
  \sup_{(t,\mathbf{x})\in[0,T]\times X}
  \mathbb D\!\bigl(
      H_J(t,\mathbf{x}),\,
      H^{r}(t,\mathbf{x})
  \bigr)
  \;\xrightarrow[J\to\infty]{\text{w.p.\;1}}\;0.
\end{equation}
 Hence, almost surely, for any $\varepsilon>0$ there exists a random index
$J_0$ such that for all $J\ge J_0$, all $\x\in X$ and any
$\v_J\in H_J(t,\x)$, one can find $\x'\in \B(\x,r)$ and $\v\in H(t,\x')$ with
$\|\v_J-\v\|\le\varepsilon$.
Then,
\[
  \|\v_J\|
  \le \|\v\| + \|\v_J-\v\|
  \le \sup_{\u\in H(t,\x')} \|\u\| + \varepsilon
  \le \widetilde L_1(t)\,\|\x'\| + \widetilde L_2(t) + \varepsilon.
\]
Because $\|\x'\|\le \|\x\| + r$, we obtain
$
  \|\v_J\|
  \le \widetilde L_1(t)\,(\|\x\|+r)
       + \widetilde L_2(t)
       + \varepsilon,
$
and taking the supremum over $\v_J\in H_J(t,\x)$ yields for any $(t,\mathbf x)\in [0,T]\times X,\
J\ge J_0,$
\begin{equation}\label{linear_H_J}
\sup_{\mathbf v\in H_J(t,\mathbf x)}\|\mathbf v\|
\;\le\;
\widetilde L_{1}(t)\|\mathbf x\|+\widetilde L_1(t)r
+\widetilde L_{2}(t)+\varepsilon.
\end{equation}
Hence, it follows
$
  \|\mathbf x_{\nu+1}^{N,J}\|
  \le(1+ah)\|\mathbf x_{\nu}^{N,J}\|+hb
$
with
$a:=\sup_{t\in[0,T]}\widetilde L_{1}(t)$ and
$b:=\sup_{t\in[0,T]}(\widetilde L_1(t)r
+\widetilde L_{2}(t)+\varepsilon)$. Then, after iteration, we have,
\begin{equation}\label{x_N}
   \|\mathbf x_{\nu}^{N,J}\|
   \;\le\;
   e^{aT}\,\|\mathbf x_{0}\|
   +\frac{b}{a}\bigl(e^{aT}-1\bigr),
   \ \nu=0,\dots,N-1.
\end{equation}
Thus, we conclude
\begin{equation}\label{Equi_bound}
\limsup_{N\to \infty}\limsup_{J \to \infty} \sup_{t \in [0,T]} \|x^{N,J}(t)\| < \infty.
\end{equation}

Moreover, by the recursive inclusion, the derivative $\dot{x}^{N,J}$ exists a.e. in $[0,T]$, and satisfies
$
\dot{x}^{N,J}(t) \in H_J(t_\nu, \x_\nu^{N,J})
$, for $t\in(t_\nu,t_{\nu+1})$, $\nu=0,\dots,N-1$. Combining (\ref{Equi_bound}) and (\ref{linear_H_J}) implies:
\begin{equation}\label{equi_dot}
\limsup_{N \to \infty}\limsup_{J \to \infty} \operatorname{sup}_{t \in [0,T]} \|\dot{x}^{N,J}(t)\| < \infty,\ \forall t\in(t_\nu,t_{\nu+1}),\ \nu=0,\dots,N-1.
\end{equation}
Hence the family $\{x^{N,J}\}_{N,J\in\N}$ is equi‐bounded and equi‐continuous, so for any $r>0$, by the Arzelà-Ascoli theorem, there are subsequences $\{N_j\}$ and  $\{J_j\}$ with $J_j\ge\max\{j,\,J_0\}$ and $N_j\to\infty$ and $J_j\to\infty$ as $j\to\infty$ such that $\{x^{N_j,J_j}\}_{j\in \N}$ converges uniformly to a function $\bar{x}$.  In addition, $\dot{\bar{x}}(\cdot)$ converges weakly in $L_\infty([0,T])^{n}$ to a limit $\bar{v}$ for some further subsequence (not relabeled). Then by (\ref{eq:ShapiroHaus}) and noting $J_j>J_0$, it follows
$
   \sup_{t\in[0,T]}\mathbb D\bigl(H_{J}(t,x^{N_j,J_j}(t)),H^{r}(t,x^{N_j,J_j}(t))\bigr)
   \le \varepsilon.
$ Moreover, from (\ref{equi_dot}), we have
\[
   \sup_{\nu\in [N_j]-1}\;
   \max_{\,t_\nu\le t\le t_{\nu+1}}
   \bigl\|x^{N_j,J_j}(t)-x^{N_j,J_j}(t_\nu)\bigr\| \to 0, \quad {\rm as} \,\, j\to\infty,
\]
and therefore, there exists a $j_1$ such that for
all $j>j_1$
\[
   \dot x^{N_j,J_j}(t)
   \in
   H_J\!\bigl(t,x^{N_j,J_j}(t)\bigr)+\,\mathbb B(\mathbf{0},\varepsilon)
   \ \text{for a.e.\ }t\in [0,T].
\]
Then it follows from the uniform convergence $x^{N_j,J_j}\to\bar x$, weak convergence
$\dot x^{N_j,J_j}\rightharpoonup\dot{\bar x}$ and the dominated convergence theorem that, for every measurable
$\Delta\subset[0,T]$ and $\q\in\bbr^{n}$ with $\|\q\|=1$,
\[
   \int_\Delta \q^\top\dot{\bar x}(\tau)\,d\tau
   \le
   \int_\Delta
      \left[\sup_{\a\in H^{\,r}(\tau,\bar x(\tau))}\q^\top \a
            +2\varepsilon\right]d\tau .
\]
Since $\Delta$ and $\varepsilon$ are arbitrary, we have for any $ r>0$,
$
   \dot{\bar x}(t)\in {\operatorname{conv}}\bigl(H^{\,r}(t,\bar x(t))\bigr)
   \ \text{a.e.\ }t\in[0,T].
$
\end{proof}

\section{Numerical experiment}\label{se:numerical-exam}
In this section, we verify our theoretical results by a numerical example. The numerical test is conducted in MATLAB 2025a on a Lenovo desktop (2.60GHz, 32.0GB RAM).

\textit{Example.} Consider problem (\ref{DSTO_first_stage})-(\ref{DSTO_second_stage}) with the following setting.  Choose two independent random variables
$
\xi_1 \sim \mathcal{N}(1, \sigma)$ and $\xi_2 \sim \mathcal{U}(-1-t, 1+t)
$ and the set
$
X = \mathbb{R}_+^2.
$

The parametric optimization problem in the DSVI-O is defined as
\begin{equation}\label{test_example}
y_i(t,\xi)\;\in\;\arg\min_{\y\in\mathcal{Y}_i(t,\xi)}
\;f_i\bigl(t,\xi,x(t),\y\bigr)
\;=\;
\frac{1}{2}\,\bigl\| H_i(t,\xi)\,\y - c_i\bigl(x(t),\xi\bigr)\bigr\|^{2}
\;+\;\mu\,\lVert \y\rVert_{1},
\end{equation}
for $i=1,2$, where $\mu>0$ is a fixed regularization parameter,
\[
H_1(t,\xi)
=
\begin{pmatrix}
t+\xi_{1}+0.5 & t+\xi_{1}+1 & \cdots & t+\xi_{1}+5 \\[3pt]
t+\xi_{2}+0.5 & t+\xi_{2}+1 & \cdots & t+\xi_{2}+5 \\[3pt]
\xi_{1}+\xi_{2}+0.5 & \xi_{1}+\xi_{2}+1 & \cdots & \xi_{1}+\xi_{2}+5
\end{pmatrix}\in \mathbb{R}^{3\times10},
\]
\[
H_2(t,\xi)
=
\begin{pmatrix}
1 & t & \xi_{1} & \xi_{2} \\[3pt]
t & t+\xi_{1} & t+\xi_{2} & \xi_{1}+\xi_{2}
\end{pmatrix}\in \mathbb{R}^{2\times4},\
\]
\[
c_1\bigl(x(t),\xi\bigr)
=
\begin{pmatrix}
x_{1}(t) + \xi_{2} \\[4pt]
x_{2}(t) + \xi_{1} \\[4pt]
\bigl(\xi_{1} + \xi_{2}\bigr)\,x_{1}(t)
\end{pmatrix} \in\mathbb{R}^{3},\
c_2\bigl(x(t),\xi\bigr)
=
\begin{pmatrix}
x_{1}(t)-x_{2}(t) \\[4pt]
\bigl(\xi_{1}-\xi_{2}\bigr)\,x_{2}(t)
\end{pmatrix}\in\mathbb{R}^{2},
\]
and the feasible sets are
$
\mathcal{Y}_1(t, \xi) = \left(1 + t + \exp(-|\xi_1|) + |\xi_2|\right)[-\mathbf{e}_1, \mathbf{e}_1]$  with $ \mathbf{e}_1 = (1,\ldots,1)^\top \in \mathbb{R}^{10}$ and $
\mathcal{Y}_2(t, \xi) = \left(2 - t + |\xi_2|\right)[-\mathbf{e}_2, \mathbf{e}_2]$ with $\mathbf{e}_2 = (1,\ldots,1)^\top \in \mathbb{R}^{4}.
$
Let $\Phi\bigl(t,\xi,x(t),y(t,\xi)\bigr)=A(t,\xi)x(t)+\sum_{i=1}^{2}B_i(t,\xi,x(t))\,y_i(t,\xi)$, where
\[
A(t,\xi)
=
\begin{pmatrix}
1+\xi_{2}+t & \xi_1-1 \\[4pt]
\xi_2   & 3+\xi_{1}
\end{pmatrix},\
B_1(t,\xi,x(t))=
\begin{pmatrix}
\bigl(2\,x_{1}(t)\,\xi_{1} + t + \xi_{2}\bigr)\,\mathbf{e}_1^{\top} \\[5pt]
\bigl(x_{1}(t)\,\xi_{2} + t + 2\,\xi_{1}\bigr)\,\mathbf{e}_1^{\top}
\end{pmatrix},
\]
\[
B_2(t,\xi,x(t))=
\begin{pmatrix}
\bigl( 2t + \xi_{2}\bigr)\,\mathbf{e}_2^{\top} \\[5pt]
\bigl( 2t + 2\,\xi_{1}\bigr)\,\mathbf{e}_2^{\top}
\end{pmatrix}.
\]

Now we illustrate that this example has a weak solution $\bigl(
x^*:[0,1] \to \mathbb{R}^2,\
y_1^*:[0,1]\times\Xi \to \mathbb{R}^{10},\
y_2^*:[0,1]\times\Xi \to \mathbb{R}^{4}
\bigr)$. We briefly verify that assumptions \ref{A1}--\ref{A4} hold.
The random vector $(\xi_1,\xi_2)$ is drawn from the product law
${P}_t=\mathcal N(1,\sigma^2)\otimes\mathcal U(-1-t,1+t)$, hence
\textbf{(A1)} is fulfilled.  Denote $\Phi_{1}(t,\xi,x(t))=A(t,\xi)x(t)$ and note
$B_i(t,\xi,x(t))$ are affine in $x(t)$. Hence, for all $\x,\x'\in\bbr^2,$
\[
\|\Phi_{1}(t,\xi,\x)-\Phi_{1}(t,\xi,\x')\|
\le \|A(t,\xi)\|\,\|\x-\x'\|,
\]
\[
\|B_1(t,\xi,\x)-B_1(t,\xi,\x')\|
\le \sqrt{4\xi_1^2+\xi_2^2}\,\|\mathbf e_1\|\,\|\x-\x'\|\;,
\]
\[
\|B_2(t,\xi,\x)-B_2(t,\xi,\x')\|=0\;.
\]
Since $\xi_2$ is bounded and $\mathbb{E}[|\xi_1|]<\infty$,
both $\|A(t,\xi)\|$ and $\sqrt{4\xi_1^2+\xi_2^2}\,\|\mathbf e_1\|$
have finite expectations uniformly in $t\in[0,T]$. Thus \textbf{(A2)} is satisfied. For fixed $x_0$, the quantities
$\mathbb E_{P_t}[\|A(t,\xi)x_0\|]$ and
$\mathbb E_{P_t}[\|B_i(t,\xi,x_0)\|]$
are uniformly bounded on $[0,1]$, since $\xi_2$ is
bounded and $\mathbb E|\xi_1|<\infty$.
Thus the corrected Assumption~\ref{A3} holds.
Finally, the feasible set $\mathcal{Y}_i(t,\xi) $ is continuous in $(t,\xi)$ and  contained in $5[-\mathbf{e}_i, \mathbf{e}_i], i=1,2$. Therefore, \textbf{(A4)} is satisfied.
Consequently, by Theorem~\ref{thm:existence} the system admits at least a
weak solution $(x^{\ast},y^{\ast})$.

To solve the inner $\ell_1$-regularized least squares problem~\eqref{test_example},
we employ the Fast Iterative Shrinkage--Thresholding Algorithm (FISTA) \cite{beck2009fast}.
In the numerical experiments, choose $\mu=5\times 10^{-3}$ and let $\hat{x} = (\hat{x}_1, \hat{x}_2)$ denote the reference numerical solution computed with a large sample size $J = 1000$ and a fixed number of time steps $N = 10^4$. For this fixed $N$, we investigate the effect of varying the standard deviation $\sigma \in \{0.1,\ 0.5,\ 1,\ 1.5\}$ and the sample size $J \in \{30,\ 100,\ 200,\ 500\}$. For each pair $(\sigma, J)$, we compute the numerical solution $\tilde{x}^{N,J}_\sigma=(\tilde{x}_{\sigma,1}^{N,J},\tilde{x}_{\sigma,2}^{N,J})$ and the accuracy
\[R_1=\frac{1}{10000}\sum_{i=1}^{10000}|\hat{x}_1(ih)-\tilde{x}_{\sigma,1}^{N,J}(ih)|,\,\,
R_2=\frac{1}{10000}\sum_{i=1}^{10000}|\hat{x}_2(ih)-\tilde{x}_{\sigma,2}^{N,J}(ih)|\]
50 times and average them. The decreasing tendencies of $R_1$ and $R_2$ as $\sigma$ decreases and $J$ increases are shown in  Figure \ref{fig8}

\begin{figure}[htb]
	\centering
	\subfigure{
		\begin{minipage}[b]{0.45\textwidth}
			\includegraphics[width=1.0\textwidth]{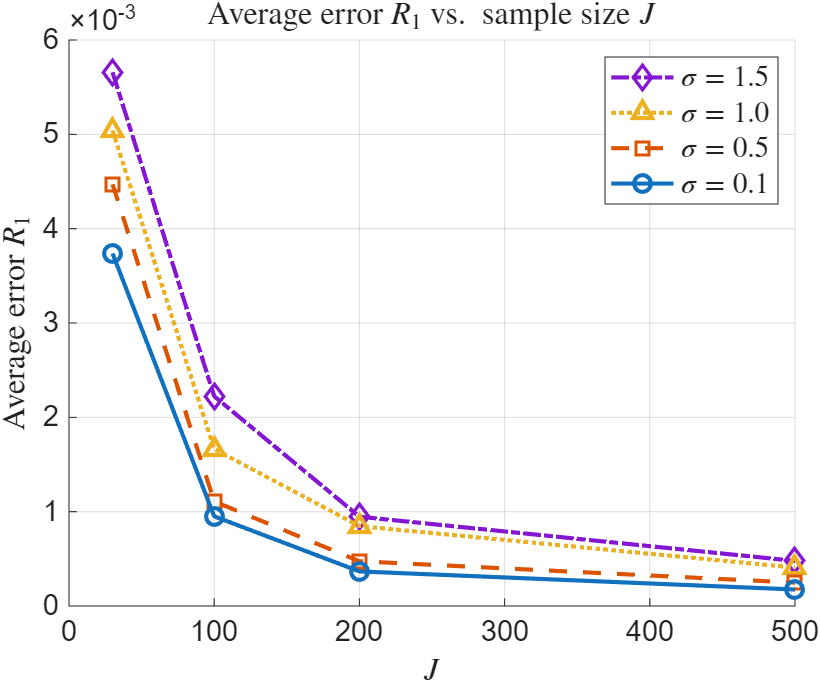}
		\end{minipage}}
\hspace{0.05in}
	\subfigure{
		\begin{minipage}[b]{0.45\textwidth}
			\includegraphics[width=1.0\textwidth]{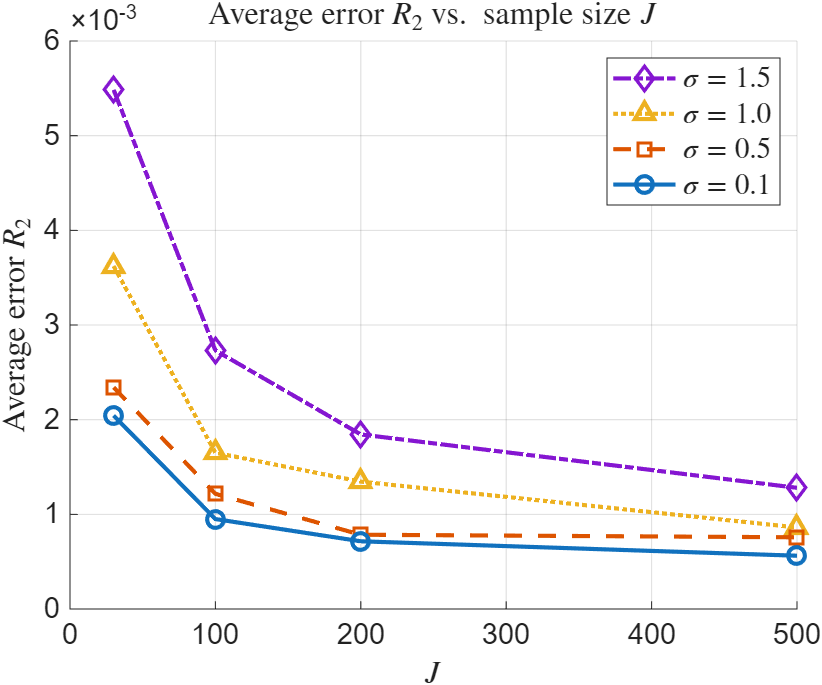}
		\end{minipage}}
		\caption{Decay of $\mathrm{R}_1$ and $\mathrm{R}_2$ with increasing $J$ and decreasing $\sigma$.}
\label{fig8}
\end{figure}

\section{Applications of Embodied Intelligence System}
In this section, we show the effectiveness of the DSVI-O with the time-stepping and SAA discrete scheme for delivering personalized health recommendations for elderly individuals aged 60 and above. The DSVI-O framework is specifically designed to provide real-time, individualized health insights through the analysis of multi-source health data, including wearable devices and medical records. Due to privacy constraints and the limited availability of continuous real-world health monitoring data, especially among elderly populations, this study utilizes synthetic data generation methodologies to rigorously evaluate the  performance of the DSVI-O with
the discrete scheme.

Synthetic data generation has gained significant recognition for its ability to replicate complex, real-world scenarios while preserving data privacy and enabling comprehensive algorithm testing \cite{giuffre2023harnessing,jordon2022synthetic,mccaw2024synthetic}. By creating diverse, physiologically plausible datasets that simulate realistic variations in elderly health conditions, the synthetic approach ensures robust, scalable validation of the proposed framework.
Figure~\ref{fig:app_DSBO} illustrates the application structure of the DSVI-O, highlighting the integration and processing of multi-source data streams to generate real-time personalized health recommendations tailored for elderly individuals.

\newcommand{\RowGap}{2.5cm}
\newcommand{\RowGapp}{2.2cm}
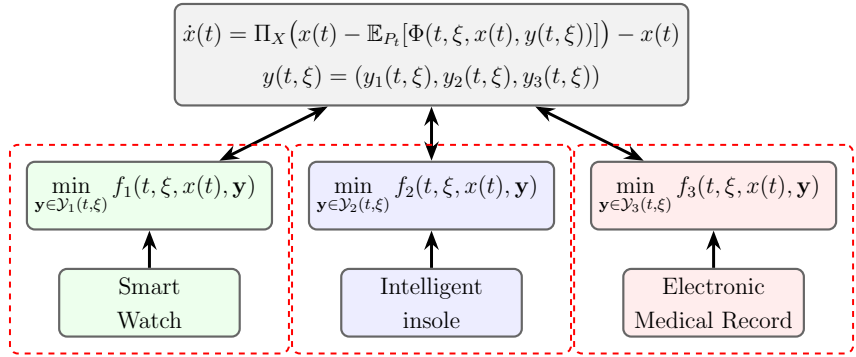
\begin{figure}[htbp]        
    \centering
    \scalebox{0.75}{  
\begin{tikzpicture}[
    >=Stealth,
    node style/.style={
        draw=black!60, line width=1.1pt,
        rectangle, rounded corners=4pt,
        align=center, inner sep=4pt
    },
    arrow/.style={-Stealth, line width=1.5pt, draw=black},
    line/.style={Stealth-Stealth, line width=1.5pt, draw=black},
    dashed group/.style={draw=red, dashed, line width=1.0pt, inner sep=8pt, rounded corners=4pt}
]

  \node[node style, fill=gray!10, minimum width=3.6cm, minimum height=1.8cm] (block0) at (0,0) {
    $\displaystyle \dot{x}(t)=\Pi_X\bigl(x(t)-\mathbb{E}_{P_t}[\Phi(t,\xi,x(t),y(t,\xi))]\bigr)-x(t)$\\[4pt]
    $\displaystyle y(t,\xi)=(y_1(t,\xi),y_2(t,\xi),y_3(t,\xi))$
  };

  \node[node style, fill=green!7, minimum width=3.2cm, minimum height=1.2cm] (block1) at (-5,-\RowGap) {
    $\displaystyle \min\limits_{\mathbf{y}\in\mathcal{Y}_1(t,\xi)} f_1(t,\xi,x(t),\mathbf{y})$
  };
  \node[node style, fill=blue!7, minimum width=3.2cm, minimum height=1.2cm] (block2) at (0,-\RowGap) {
    $\displaystyle \min\limits_{\mathbf{y}\in\mathcal{Y}_2(t,\xi)} f_2(t,\xi,x(t),\mathbf{y})$
  };
  \node[node style, fill=red!7, minimum width=3.2cm, minimum height=1.2cm] (block3) at (5,-\RowGap) {
    $\displaystyle \min\limits_{\mathbf{y}\in\mathcal{Y}_3(t,\xi)} f_3(t,\xi,x(t),\mathbf{y})$
  };

  \node[node style, fill=green!7, minimum width=3.2cm, minimum height=1.2cm] (data1) at (-5,-2*\RowGapp) {Smart\\ Watch};
  \node[node style, fill=blue!7, minimum width=3.2cm, minimum height=1.2cm] (data2) at (0,-2*\RowGapp)      {Intelligent\\ insole};;
  \node[node style, fill=red!7, minimum width=3.2cm, minimum height=1.2cm] (data3) at (5,-2*\RowGapp)      {Electronic\\ Medical Record};

  \draw[line] (block0) -- (block1);
  \draw[line] (block0) -- (block2);
  \draw[line] (block0) -- (block3);
  \draw[arrow] (data1) -- (block1);
  \draw[arrow] (data2) -- (block2);
  \draw[arrow] (data3) -- (block3);

  \node[dashed group] (group1) [fit=(block1)(data1)] {};
  \node[dashed group] (group2) [fit=(block2)(data2)] {};
  \node[dashed group] (group3) [fit=(block3)(data3)] {};
\end{tikzpicture}}
  \caption{DSVI-O workflow for personalized elderly health recommendations}  
    \label{fig:app_DSBO}
\end{figure}

\subsection{Data Generation}
The generated data encompass physiological parameters, clinical records, and wearable device outputs, such as insole sensor readings. This comprehensive yet controlled synthetic approach facilitates rigorous and scalable evaluation of the DSVI-O framework while effectively addressing data privacy and scarcity constraints.

\subsubsection{Hybrid Generation of Synthetic Health Data}

Traditional rule-based methods lack flexibility and realism, while deep learning approaches typically require large datasets and yield less interpretable results~\cite{rujas2025synthetic}. To address these limitations, our hybrid approach combines rule-based constraints, knowledge distillation from Large Language Models (LLMs), and stochastic mathematical modeling to ensure both realism and interpretability. Specifically, we employ:

\begin{enumerate}
\item Physiological rules (e.g., BMI sampled from $\mathcal{N}(22, 2.5^2)$, clipped to $[18.5, 30]$),
\item LLM (Qwen2.5-7B-Instruct) for flexible, guideline-based rule generation,
\item Stochastic mathematical modeling, using Gaussian distributions centered on clinical medians~\cite{halmich2025data}.
\end{enumerate}

\paragraph{Individual Attributes}
Age groups are 60–69 (50\%), 70–79 (30\%), 80+ (20\%). Gender ratio is 1:1. {Height follows age-adjusted distributions: $\mathcal{N}(170-0.1(\text{age}-60), 5^2)$ for men, $\mathcal{N}(160-0.1(\text{age}-60), 5^2)$ for women. Values are clipped to the range [145 cm, 190 cm]. Weight is calculated using the formula $\text{BMI} \times (\text{height}/100)^2$.}

\paragraph{Clinical Indicators}
Static EMRs (e.g., WBC, RBC, HGB, ...) are sampled from normal distributions around clinical medians, adjusted for age, gender, and BMI, within clinical limits.

\paragraph{Health State Groups}
Users are categorized into healthy, weak or ill groups (e.g., heart rate, SpO$_2$, sleep, steps, and calories), with specific indicator shifts (e.g., heart rate group: CRP +2, WBC +1.5), clipped within medical ranges.

\paragraph{Physiological Data}
The physiological data were generated at 5-second intervals, incorporating circadian rhythms, activity, and health status. The baseline heart rate was defined as
$
\text{HR} = 70 - 0.2(\text{age} - 60) + 2\,\mathbb{I}_{\{\text{male}\}} + 0.05(\text{SBP} - 120) + 0.01(\text{HGB} - 140),
$
and modulated by $\pm 8$ bpm based on sleep or activity state.

\paragraph{Insole Data}
The insole data were generated at 5-second intervals and include 17 indicators such as plantar pressures on four channels (left forefoot, left heel, right forefoot, right heel), step frequency, stride length, left–right symmetry, force estimation, and balance status. These indicators are influenced by factors such as weight, age, gait pattern, and abnormal conditions, with added Gaussian noise and values constrained within physiologically realistic ranges.

\subsubsection{Overview of the Generated Dataset}
{The test dataset consists of 10 individuals, each with
10 days of randomly generated testing data.
For each individual, the experiments also use 90 days
of historical data. The possible health states for each individual include: healthy ({\color{green}0}), weak ({\bf\textcolor[rgb]{1,0.8,0.2}{1}}), and ill ({\color{red}2}), as well as transitional states between these categories, such as 0.1 or 1.6.
Each individual's daily health status evolves according to one of four temporal patterns: jump (sudden transitions, e.g., acute events), linear (gradual change, e.g., fatigue accumulation), exponential (rapid onset with saturation, e.g., drug response), and logistic (slow–fast–slow transitions, e.g., chronic recovery). These patterns describe transitions between health states 
over each 24-hour period. Labels are recorded every five seconds, yielding fine-grained time-series trajectories. Such structured modeling of heterogeneous transitions aligns with biomedical insights on personalized health trajectories \cite{jorgensen2024disease}.

Each person is represented by a series of features organized into multiple categories, each stored in a separate CSV file. All generated data are publicly available on GitHub at:
 \texttt{https://github.com/DSVI2025/DSVI-O}.

\begin{itemize}
    \item \texttt{user\_profiles.csv}: demographic attributes (user ID, age, gender, height, weight, BMI).
    \item \texttt{user\_groups.csv}: health trajectory types and group labels.
    \item \texttt{medical\_records.csv}: blood counts, metabolic panels, blood pressure, etc.
    \item \texttt{health\_data.csv}: daily metrics including heart rate, SpO$_2$, sleep quality, steps, calories, {skin temperature, hrv, stress index, sbp, dbp, activity intensity, active minutes, floors climbed activity, and status.}
    \item \texttt{insole\_data.csv}: foot pressure, {motion types,} and gait information from wearable insole devices.
\end{itemize}

\subsection{Data-Driven Health State Modeling via DSVI-O}

We apply DSVI-O ~\eqref{DSTO_first_stage}--\eqref{DSTO_second_stage} to dynamically model the evolution of an individual's health status over 24 hours. The first-stage state variable \( x(t) \in \mathbb{R}^n \) denotes the personalized health state of an elderly individual at time \( t \), evolving continuously under the influence of real-time monitoring data. The second stage incorporates inputs from three heterogeneous sources: smart watch signals, intelligent insole data, and EMR. Historical data from each source is used to continuously calculate the time-varying parameters in the ODE, enabling real-time tracking and personalized updates of the health trajectory.

To capture the multi–factor evolution of an elderly individual’s health quality, we model the latent state by an ODE assembled from three physiologically motivated components.

\paragraph{1. Self-dynamics model}
Following the frailty index framework in~\cite{rockwood2007frailty},  the self-evolution of an individual’s health status through an ODE is modeled as
$$\dot{x}(t)
   = -A(t)x(t) -p(t),$$
where $A(t)$ and $p(t)$ are damage-repair parameters.

\paragraph{2. Circadian rhythm (CR) model}

Wearable sensors now provide high-resolution, non-invasive activity traces, enabling large-scale modelling of individual CR patterns. We embed this circadian component in the ODE ~\eqref{DSTO_first_stage} to reflect
the 24-hours signal by a truncated Fourier series \cite{kim2023efficient}:
$$q(t,\xi)=
   \sum_{r=1}^{M} A_r(t)\,
   \sin\!\bigl(2\pi\zeta_r t+\phi_r(t)\bigr)
   +\varepsilon_{1}(t,\xi),$$
where
\(\phi_r(t)\) and \(A_r(t)\) are the phase and amplitude of the \(r\)-th harmonic, \(\varepsilon_1(t,\xi)\) aggregates noise and
\(\zeta_r\) sets the circadian frequencies.

\paragraph{3. Latent exogenous drive}
The health state can be perturbed by external factors such as physical activity, gait stability, or clinical interventions.
We represent these effects through a \emph{latent exogenous signal}
$\ell(t,\xi)=\ell_{0}+\Delta\ell_{1}(t)\,m(t,\xi),$
where $\ell_{0}$ is the long-run baseline and
$\Delta\ell_{1}(t)$ captures time-varying departures.  The composite score $m(t,\xi)$ is built from the latest
wearable-sensor read-outs via a linear fusion model
\[
m(t,\xi)=\sum_{i=1}^{3}\alpha_{i}(t)\,
         B_{i}(t,\xi)\,y_{i}(t,\xi)+\varepsilon_{2}(t,\xi),
\]
with $B_{i}(t,\xi)$ the $i$-th physiological channel,
$y_{i}(t,\xi)$ the physiological parameter, and $\alpha_{i}(t)$ an adaptive
weight.

Building upon the three components above, the latent health state \(x(t) \in X\) is modeled as the following
\begin{align}\label{eq:full_ode}
\dot{x}(t)=
\Pi_{X}\!\Bigl(
      x(t)-{A(t)\,x(t)-p(t)+\mathbb{E}_{P_t}\bigl[q(t,\xi)+\ell(t,\xi)\bigr]}
                     \Bigr)
      -x(t).
\end{align}
To evaluate \(m(t,\xi)\) in  \(\ell(t,\xi)\), we  solve the following constrained quadratic program:
\begin{align}\label{eq:full_opt}
y_i(t,\xi)\in
\mathop{\arg\min}_{\y\in[a_i(\xi),\,b_i(\xi)]}
f_i\bigl(t,\xi,x(t),\y\bigr), \quad i=1,2,3,
\end{align}
where
$
f_i(t,\xi,\x,\y)=
\tfrac{1}{2}
\bigl\|H_i(t,\xi)\,\y+N_i(t,\xi)\,\x-c_i(t,\xi)\bigr\|^{2}
+\tfrac{\rho_i}{2}\bigl\|B_i(t,\xi)\,\y-\x\bigr\|^{2}.
$
Here, the first term ensures that the model follows clinical priors, while the second term keeps the estimated state close to physiological features. These parameter estimates are then used to compute \(m(t,\xi)\) and later \(\ell(t,\xi)\), forming a feedback loop in the projected ODE.
In addition, the matrices \(H_i(t,\xi) \in \mathbb{R}^{n_i \times m_i}\) and \(c_i(t,\xi) \in \mathbb{R}^{n_i}\) are built from the individual's historical data, where \(H_i\) collects past wearable and physiological signals, and \(c_i\) gives the related clinical labels. The correction matrix \(N_i(t,\xi) \in \mathbb{R}^{n_i \times n}\) helps reduce the effect of past data that are not similar to the current state.

The following conditions and Assumption \ref{A1} ensure the existence of a weak solution to \eqref{eq:full_ode}--\eqref{eq:full_opt}.

The following conditions, together with Assumption~\ref{A1}, ensure the existence
of a weak solution to \eqref{eq:full_ode}--\eqref{eq:full_opt}.

\begin{enumerate}[label=\textbf{(B\arabic*)},leftmargin=0.35in]
\item \label{B1}
The functions \(A(\cdot)\), \(p(\cdot)\), \(\Delta\ell_1(\cdot)\),
\(\alpha_i(\cdot)\), \(i=1,2,3\), \(A_r(\cdot)\), and \(\phi_r(\cdot)\),
\(r=1,\ldots,M\), are continuous on \([0,T]\).

\item \label{B2}
The mappings \(\varepsilon_j\), \(j=1,2\), and \(B_i\), \(i=1,2,3\), are continuous, and
$
\mathbb E_{P_t}[\|\varepsilon_j(t,\xi)\|]$ and $
\mathbb E_{P_t}[\|B_i(t,\xi)\|]
$
are uniformly bounded in \(t\in[0,T]\), and hence integrable on \([0,T]\).

\item \label{B3}
The functions \(a_i,b_i:\Xi\to\mathbb R^{m_i}\), \(i=1,2,3\), are continuous
and uniformly bounded, and \(a_i(\xi)\le b_i(\xi)\) componentwise for all
\(\xi\in\Xi\).
\end{enumerate}

It is easy to verify that \textbf{(A2)}--\textbf{(A4)} hold in the present example. Indeed, by the definitions of \(q\), \(\ell\), and \(m\), \textbf{(A2)} holds with
$
\tilde\kappa_{\Phi_1}(t,\xi)=\|A(t)\|,
\ 
\tilde\kappa_{B_i}(t,\xi)=0,\  i=1,2,3.
$
Moreover, \textbf{(A3)} holds with
$
\bar\beta
:=
\sup_{t\in[0,T]}
\bigl\{
\|A(t)x_0\|+\|p(t)\|+\|\ell_0\|
+\mathbb E_{P_t}[\|q(t,\xi)\|]
+|\Delta\ell_1(t)|\mathbb E_{P_t}[\|\varepsilon_2(t,\xi)\|]
+|\Delta\ell_1(t)|
\max_{i\in[3]}|\alpha_i(t)|\mathbb E_{P_t}[\|B_i(t,\xi)\|]
\bigr\}<\infty .
$
Finally, \textbf{(B3)} implies \textbf{(A4)}. Thus, by
Theorem~\ref{thm:existence}, \eqref{eq:full_ode}--\eqref{eq:full_opt}
has at least one weak solution on \([0,T]\) for any \(T>0\).

\subsection{Numerical Experiment Setup}
We use the DSVI-O model (\ref{eq:full_ode})-(\ref{eq:full_opt}) to simulate the evolution of the health status of 10 persons over 10 days by using their health data in the past 90 days. Within every 24-hour cycle, measurements were sampled at 5-second intervals, yielding 17,280 time points per day. We consider three types of multimodal data: smart watch signals (14 features), intelligent insole measurements (17 features), and electronic medical records (30 features) as illustrated in Figure \ref{fig:app_DSBO}.

For each person, the following three matrices and three vectors contain his/her health features and corresponding labels for $\nu=1, \ldots, 17280$:
\begin{itemize}
  \item \texttt{smart watch} \hspace*{4.6em}
  $Z_1(t_\nu) \in \mathbb{R}^{100 \times 14}, \quad \eta_1(t_\nu) \in \mathbb{R}^{100}$

  \item \texttt{intelligent insole} \hspace*{1em}
  $Z_2(t_\nu) \in \mathbb{R}^{100 \times 17}, \quad \eta_2(t_\nu) \in \mathbb{R}^{100}$

  \item \texttt{medical records} \hspace*{2.6em}
  $Z_3(t_\nu) \in \mathbb{R}^{100 \times 30}, \quad \eta_3(t_\nu) \in \mathbb{R}^{100}$.
\end{itemize}

At each time $t_\nu$, from the first 90 rows of $Z_i(t_\nu)$, we randomly and independently chose 50 rows 20 times.  This yields sampled matrices $H_i\left(t_\nu, \xi^\iota\right) \in \mathbb{R}^{50 \times m_i}$ and corresponding labels $c_i\left(t_\nu, \xi^\iota\right) \in \mathbb{R}^{50}$,  $\iota=1, \ldots, 20$ and $i=1,2,3$ with $m_1=14$, $m_2=17$ and $m_3=30.$

We simulate the health dynamics of a single person over one day.
For each day $j\in[10]$, we extract the $j$-th day's data for the person
$
B_i^j(t_\nu) := \left(Z_i(t_\nu)\right)_{[90 + j,\,:]}$, $i\in[3],\ \nu\in[17280].$
This process is repeated across $j=1,\dots,10$ to obtain 10 simulated trajectories for the same person.

To reduce the influence of dissimilar historical data, we define a correction vector \(N_i^j(t_\nu, \xi^\iota) \in \mathbb{R}^{50}\) based on softmax-normalized Gaussian similarities. Specifically, given the test feature vector \(B_i^j(t_\nu)\) and the historical data matrix \(H_i(t_\nu, {\xi}^\iota)\), we compute
\begin{equation} \label{eq:correction_matrix}
N_i^j(t_\nu, \xi^\iota) = \lambda_i \left( \mathbf{1}_{50} -
\texttt{softmax} \left( -\frac{1}{h_i^2} \cdot
\texttt{pdist2}\left( B_i^j(t_\nu),\, H_i(t_\nu, {\xi}^\iota) \right)
\right) \right),
\end{equation}
with \(h_i\) being defined as the median of all pairwise distances between rows of \(H_i(t_\nu, \xi^\iota)\):
$$h_i = \operatorname{median} \left\{ \| (H_i(t_\nu, \xi^\iota))_{[r,:]} - (H_i(t_\nu, \xi^\iota))_{[s,:]} \| \;\middle|\; 1 \le r < s \le 50 \right\}.$$
The \texttt{pdist2} computes a vector in $\bbr^{50}$ whose the $r$th component is the squared Euclidean distance between \(\mathbf{B} \in \mathbb{R}^{1 \times m_i}\) and the $r$th row of \(\mathbf{H} \in \mathbb{R}^{50 \times m_i}\):
$
(\texttt{pdist2}(\mathbf{B}, \mathbf{H}))_{r} = \sum_{k=1}^{m_i}
\left( \mathbf{B}_{[1,k]} - \mathbf{H}_{[r,k]} \right)^2,
$ $r=1,\ldots,50$.
For a vector \(\mathbf{v} \in \mathbb{R}^{50}\), the \texttt{softmax}$(\mathbf{v}) \in \mathbb{R}^{50}$ is defined by
$
(\texttt{softmax}(\mathbf{v}))_{r} =
\frac{e^{v_r}}{\sum_{k=1}^{50} e^{v_k}},
$ $r=1,\ldots,50$.

The correction term introduces a similarity-dependent
adjustment to the historical regression targets.

All simulations use the following parameters:
$A(t) \!\equiv\! \exp(-3)$,
$p(t) \!\equiv\! 0.01$,
$q(t,\xi) = q(t)
= 0.15 \sin\!\bigl(2\pi \cdot 0.025t\bigr)
 + 0.17 \sin\!\bigl(2\pi \cdot 0.021t\bigr),$
$\varepsilon_{2}(t,\xi) \equiv 0,$
$\ell_0 = 0.02$,
$\Delta \ell_1(t) \equiv 1.5$,
$\alpha_1(t)\equiv 0.4,$
$\alpha_2(t)\equiv 0.4,$
$\alpha_3(t)\equiv 0.2,$
$\lambda_i = 0.50$, and
$\rho_i = 5$ for $i = 1,2,3$.
The initial state is defined as $\mathbf{x}_0^j = \eta_1(t_0)_{[90+j]}$ for each $j \in [10]$.
Overall, for each day \(j \in[10]\) and time step \(\nu = 0,\dots,17279\), the latent health state \(\mathbf{x}_\nu^j \in [0,2]\) evolves as
\begin{equation} \label{eq:health_dynamics}
\mathbf{x}_{\nu+1}^j =
\Pi_{[0,2]}\!
\left(
      \mathbf{x}_\nu^j
      - A(t_\nu)\,\mathbf{x}_\nu^j-p(t_\nu)+q(t_\nu)
      + \frac{1}{20}\sum_{\iota=1}^{20}\ell^j(t_\nu,{\xi}^\iota)
\right),
\end{equation}
where \(\Pi_{[0,2]}\) denotes projection onto the interval \([0,2]\), and
\[
\ell^j(t_\nu,{\xi}^\iota)=\ell_{0}+\Delta\ell_{1}(t_\nu)\,\left(\sum_{i=1}^{3}\alpha_{i}(t_\nu)\,
         B^j_{i}(t_\nu)\,\y_{i}^j(t_\nu,{\xi}^\iota)+\varepsilon_{2}(t_\nu,\xi^\iota)\right).
\]
At each \(t_\nu\), the parameter \(\mathbf{y}_i^j(t_\nu,\xi^\iota)\) for \(i=1,2,3\) solves
\begin{equation*}
\begin{aligned}
\min_{\mathbf{y}_i \in [-10,10]^{m_i}} \frac{\rho_i}{2} \left\| B_i^j(t_\nu) \mathbf{y}_i\! -\! \mathbf{x}_\nu^j \right\|^2 \! +\! \frac{1}{2} \left\| H_i(t_\nu,{\xi}^\iota) \mathbf{y}_i
\!+\! N_i^j(t_\nu,{\xi}^\iota)\mathbf{x}_\nu^j
\!-\! c_i(t_\nu,{\xi}^\iota) \right\|^2.
\end{aligned}
\end{equation*}
Then, the predicted health state is discretized as
\begin{equation} \label{eq:health_prediction}
\text{Pred}_\nu^j =
\begin{cases}
0, & \text{if } \mathbf{x}_\nu^j \le \frac{2}{3}, \\
1, & \text{if } \frac{2}{3} < \mathbf{x}_\nu^j < \frac{4}{3}, \\
2, & \text{if } \mathbf{x}_\nu^j \ge \frac{4}{3}.
\end{cases}
\end{equation}
In addition, the real health state trajectories are set as:
\begin{equation} \label{eq:health_dynamics_true}
\tilde{\mathbf{x}}_{\nu+1}^{j} =
\Pi_{[0,2]}\!
\left(
      \tilde{\mathbf{x}}_\nu^j
      - A(t_\nu)\,\tilde{\mathbf{x}}_\nu^j-p(t_\nu)+q(t_\nu)
      + \ell_0+\Delta\ell_1(t_\nu)\sum_{i=1}^{3}\alpha_i(t_\nu)\eta_i(t_\nu)_{[90+j]}
\right).
\end{equation}

Finally, to assess generalizability, the entire procedure above is replicated for 10 different persons.

\subsection{Numerical Experiment Result}
Following the experimental setup described above, we now present the evaluation outcomes. {The numerical tests are implemented in Python 3.12 on an Apple M3 (4.05 GHz, 16 GB RAM).}
{We adopt an individual-wise testing strategy to evaluate the last 10 days of data for each of the 10 individuals in the test set separately.} Performance metrics are computed at the individual level and then averaged to obtain the overall performance. This dual perspective allows us to assess both personalized and aggregate behavior of the DSVI-O.
\subsubsection{Performance Metrics}

To quantitatively assess the performance of our classification model, we employ a comprehensive set of evaluation metrics: \textbf{Accuracy}, \textbf{Precision}, \textbf{Recall}, \textbf{Specificity}, and \textbf{F1-score}~\cite{hossin2015review}. These metrics are computed based on the following quantities for each class $i$ ($i = 0, 1, 2$):
\begin{itemize}
  \item \textbf{True Positives ($TP_i$)}: Number of samples correctly predicted as class $i$.
  \item \textbf{True Negatives ($TN_i$)}: Number of samples correctly predicted as not belonging to class $i$.
  \item \textbf{False Positives ($FP_i$)}: Number of samples incorrectly predicted as class $i$.
  \item \textbf{False Negatives ($FN_i$)}: Number of samples incorrectly predicted as not class $i$.
\end{itemize}

Based on these definitions, we compute the following \emph{per-class metrics}, each quantifying a distinct aspect of model performance:

\begin{itemize}
  \item \textbf{Accuracy}: Measures the overall proportion of correctly classified instances among all predictions
\[
    \text{Accuracy}(i) = \frac{TP_i + TN_i}{TP_i + TN_i + FP_i + FN_i}.
\]

  \item \textbf{Precision}: Indicates the proportion of positive predictions that are actually correct
\[
    \text{Precision}(i) = \frac{TP_i}{TP_i + FP_i}.
\]

  \item \textbf{Recall}: Represents the proportion of actual positives correctly identified
\[
    \text{Recall}(i) = \frac{TP_i}{TP_i + FN_i}.
\]

  \item \textbf{Specificity}: Quantifies the proportion of actual negatives correctly identified
\[
    \text{Specificity}(i) = \frac{TN_i}{TN_i + FP_i}.
\]

  \item \textbf{F1-score}: Provides a harmonic mean of precision and recall
\[
    F1(i) = \frac{2 \cdot \text{Precision}_i \cdot \text{Recall}_i}{\text{Precision}_i + \text{Recall}_i}.
\]
\end{itemize}

\subsubsection{Individual-wise Performance}

Figure~\ref{fig:matrix} presents the confusion matrices for one person across ten testing days, offering a detailed view of the DSVI-O framework's performance at the subject level. Accuracy remains consistently above 96\% throughout. 
 Most errors occur when the true label is \textit{weak} ({\color[rgb]{1,0.8,0.2}1}) but predicted as \textit{healthy} ({\color{green}0}), suggesting that the model is more conservative in identifying early or mild degradations. On Day 9, the main misclassification pattern shifts, with \textit{healthy} ({\color{green}0}) more often predicted as \textit{weak} ({\color[rgb]{1,0.8,0.2}1}), possibly due to temporal variations in physiological signals. While the model reliably captures major health trends, day-specific dynamics may influence its sensitivity to borderline states, especially during transitions.

  \begin{figure}
    \centering
    \includegraphics[width=1\textwidth]{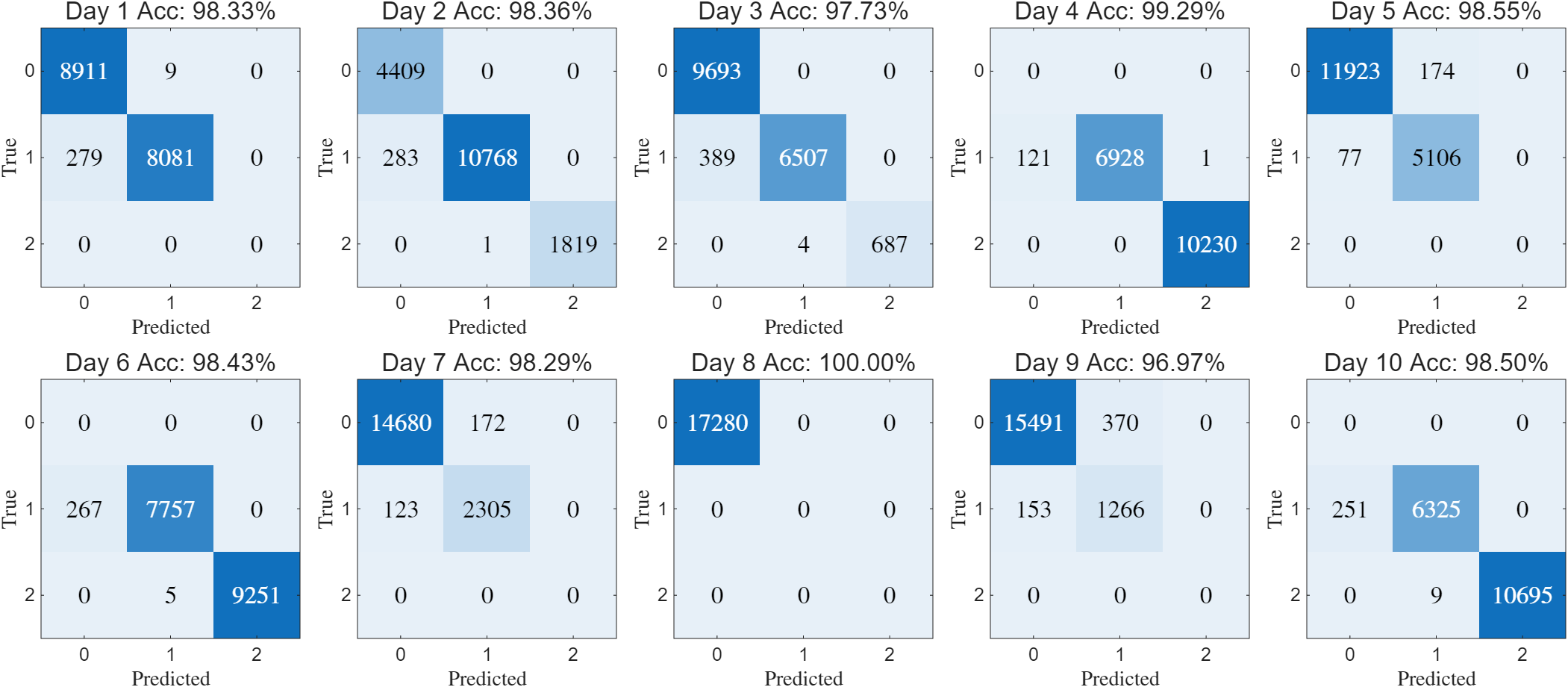}
    \caption{Confusion matrix comparing true and predicted labels for one person in 10 days.}
  \label{fig:matrix}
  \end{figure}

Figure~\ref{fig:curve} illustrates the health state trajectories $\tilde{\x}_\nu^j$ in \eqref{eq:health_dynamics_true}  and the real-time predictions $\x_\nu^j$ in \eqref{eq:health_dynamics} for one individual over 10 representative days in the test set. For each day, only a local segment containing a state transition is shown. The horizontal axis denotes time, allowing for fine-grained comparison between the true and predicted health states. From these trajectories, we observe that the model responds well to changes in physiological signals and accurately captures transitions among \textit{healthy}, \textit{weak}, and \textit{ill} states. In general, the model tracks health trends effectively, with most predicted trajectories aligning well with ground truth. Sudden state changes are typically detected in real time. In contrast, for gradual transitions, the model often fails to react immediately when the true state begins to shift but successfully identifies the change once the deterioration accumulates. This behavior reflects the model's reliance on temporal evidence aggregation when detecting slow-onset degradations.

  \begin{figure}
    \centering
    \includegraphics[width=0.9\textwidth]{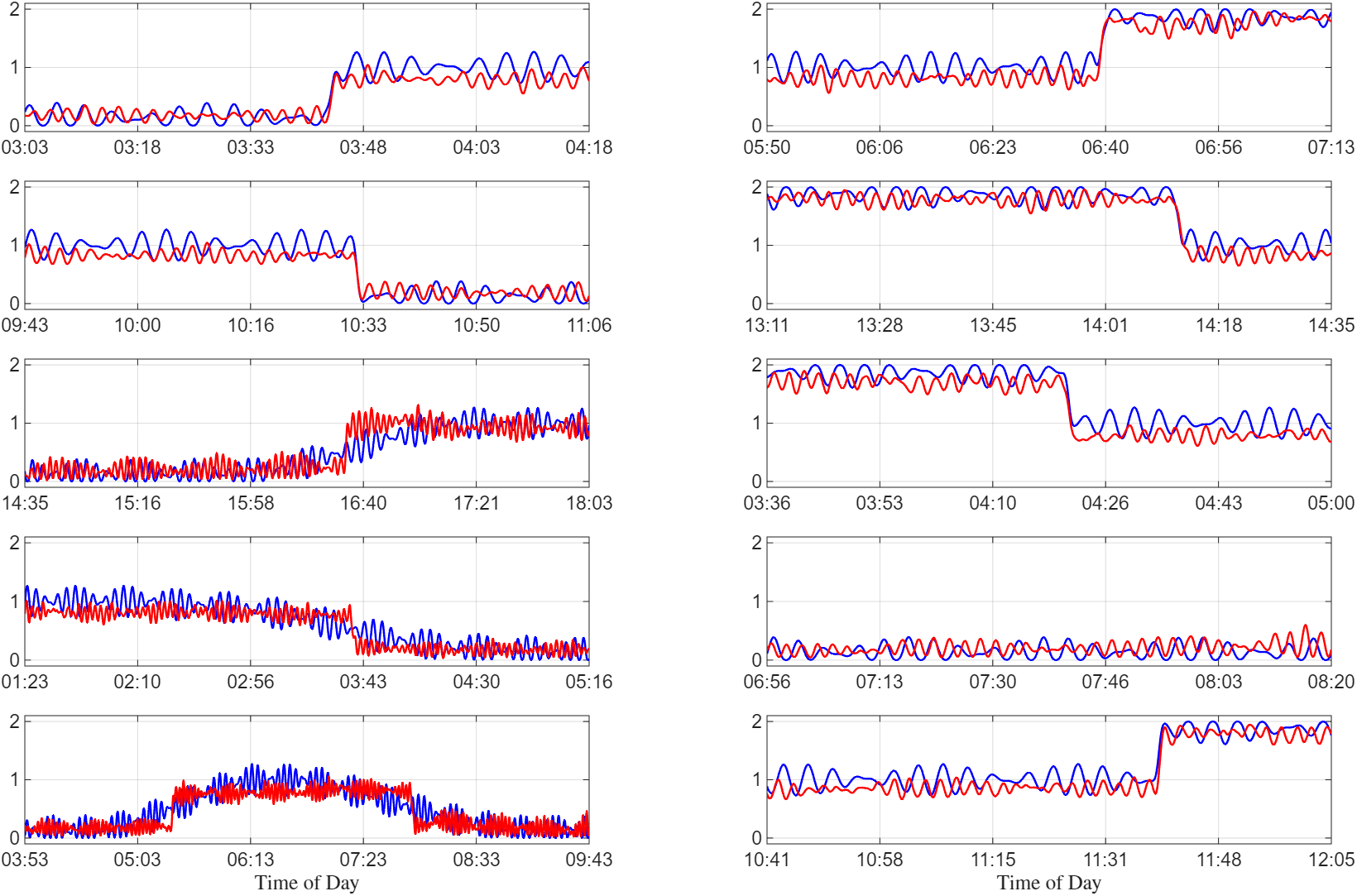}
\caption{Tracked health trajectories for one person:
estimated ({\color{red}red}) $\x_\nu^j$ versus true
({\color{blue}blue}) $\tilde{\x}_\nu^j$.}
  \label{fig:curve}
  \end{figure}

\subsubsection{Overall Performance}

Figure~\ref{fig:matrix2} presents the aggregated confusion matrix for all 10 individuals in the test set, offering a comprehensive summary of the model's classification performance across the three health states—\textit{healthy} ({\color{green}0}), \textit{weak} ({\color[rgb]{1,0.8,0.2}1}), and \textit{ill} ({\color{red}2}). The diagonal entries indicate correctly classified instances, while off-diagonal elements represent misclassifications. Over a total of 1,728,000 time points (10 individuals over 100 days), the model achieves an overall accuracy of 98.18\%. Among all predictions, only 53 instances of class \textit{healthy} ({\color{green}0}) are incorrectly predicted as \textit{ill} ({\color{red}2}), and no instances of class \textit{ill} ({\color{red}2}) are misclassified as \textit{healthy} ({\color{green}0}). Since these two classes lie at opposite ends of the health spectrum with \textit{weak} ({\color[rgb]{1,0.8,0.2}1}) in between, this result highlights the model’s robustness in distinguishing distant health conditions and avoiding severe misclassifications.

 \begin{figure}
    \centering
    \includegraphics[width=0.75\textwidth]{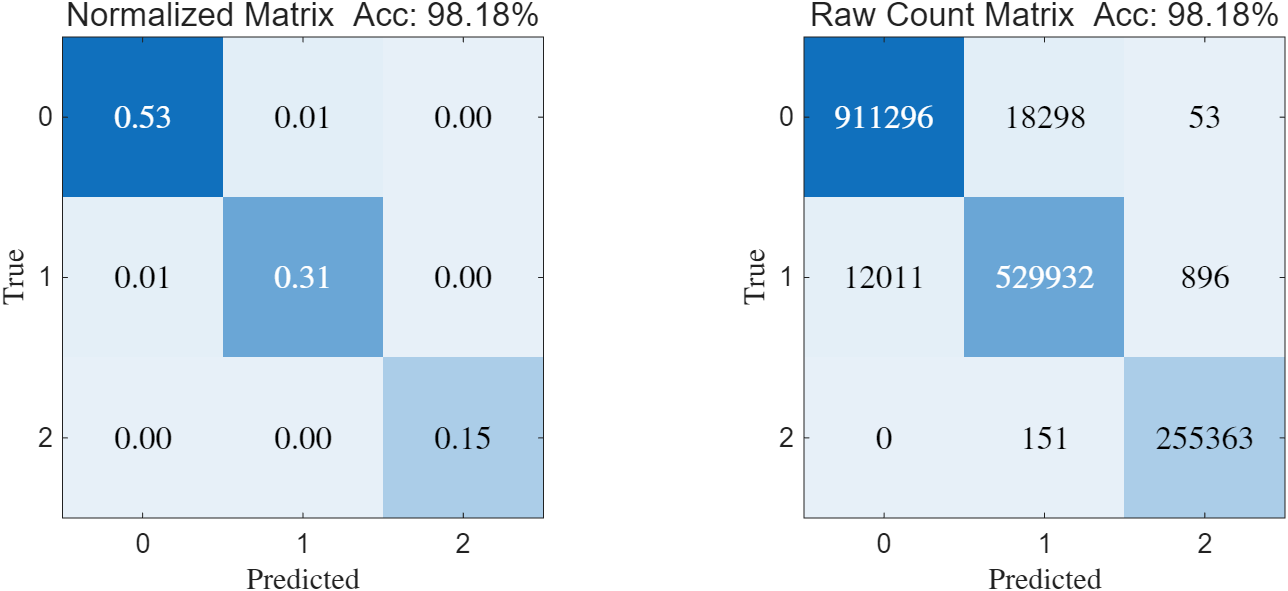}
    \caption{Confusion matrix of true and predicted labels for ten persons.}
  \label{fig:matrix2}
  \end{figure}

Table~\ref{tab:combined-eval} reports the class-wise evaluation metrics for the three health states: \textit{healthy} ({\color{green}0}), \textit{weak} ({\color[rgb]{1,0.8,0.2}1}), and \textit{ill} ({\color{red}2}). The results demonstrate that the model achieves strong performance across all classes, with overall accuracy above 98\% for each. In particular, the model performs exceptionally well in identifying the \textit{ill} state, achieving a Precision of 0.9963, Recall of 0.9994, and an F1-score of 0.9979. This indicates both high sensitivity and specificity in detecting critical health deterioration. The \textit{healthy} class also shows robust metrics, with an F1-score of 0.9836. In contrast, the \textit{weak} class exhibits slightly lower Precision (0.9664) and Recall (0.9762), suggesting that transitional or mild health changes may be more difficult to detect. These observations highlight a potential direction for model refinement, particularly in improving sensitivity to intermediate health states.

\begin{table}[ht]
\centering
\caption{Class-wise evaluation results for ten persons}
\label{tab:combined-eval}
\begin{tabular}{lccccc}
\toprule
\textbf{Type (i)} & \textbf{Accuracy} & \textbf{Precision} & \textbf{Recall} & \textbf{Specificity} & \textbf{F1-Score} \\
\midrule
healthy ({\color{green}0})          & 0.9824 & 0.9870 & 0.9803 & 0.9850 & 0.9836 \\
weak ({\color[rgb]{1,0.8,0.2}1})    & 0.9819 & 0.9664 & 0.9762 & 0.9844 & 0.9713 \\
ill ({\color{red}2})                & 0.9994 & 0.9963 & 0.9994 & 0.9994 & 0.9979 \\
\bottomrule
\end{tabular}
\end{table}

\section{Conclusions}\label{se:conclusions}

This paper systematically studies fundamental solution properties of the differential stochastic variational
inequality with parametric convex optimization (DSVI-O) \eqref{DSTO_first_stage}-\eqref{DSTO_second_stage}.  The DSVI-O  extends (DSVI), (DVI), (SVI) and (OCODE) by possessing the time-dependent distributions and handling data from different sources in
different forms simultaneously.
 We prove that the DSVI-O has at least one weak solution under
standard assumptions. Moreover, we prove that the DSVI-O has at least one classical solution when the objective
functions in the parametric convex optimization problems \eqref{DSTO_second_stage} are strongly convex. The
uniqueness of the classical solution of the DSVI-O is shown in Remark~\ref{remark_unique} for a special class of
strongly convex unconstrained parametric optimization problems in \eqref{DSTO_second_stage}. To find a numerical
solution, we propose a discrete scheme of DSVI-O by using a time-stepping approximation and the SAA. We prove
the convergence of the discrete scheme, and  establish an explicit   Euler
discretization error bound with $O(h)$ order on the time grid. We apply the DSVI-O with the discrete scheme to an embodied
intelligence system for elderly health and generate synthetic health care data by Multimodal Large Language
Models to test the performance. Preliminary numerical results demonstrate the efficiency and effectiveness of
the DSVI-O. We will study \eqref{DSTO_first_stage}-\eqref{DSTO_second_stage} with nonconvex objective functions
$f_i$ from neural network training in the further research. Such extension needs new theory and algorithms for handling parametric nonconvex optimization problems in a dynamic and stochastic system.

\vspace{0.02in}

{\bf Acknowledgments.}
We are grateful to Prof. Terry Rockafellar for his helpful discussions on
solution properties of parametric convex optimization problems and for his
comments on the Lipschitz continuity of strongly convex unconstrained
optimization problems in Remark~\ref{remark_unique}. We also sincerely thank
the Associate Editor and the referees for their careful reading and constructive
comments, which helped improve the quality of this paper.

\bibliographystyle{siamplain}
\bibliography{references1}

@article{artstein1989parametrized,
  author  = {Artstein, Z.},
  title   = {Parametrized integration of multifunctions with applications to control and optimization},
  journal = {SIAM J. Control Optim.},
  year    = {1989},
  volume  = {27},
  number  = {6},
  pages   = {1369--1380},
  doi     = {10.1137/0327070}
}

@book{aubin2009differential,
  title={Differential Inclusions},
  author={Aubin, J.-P. and Cellina,  A.},
  year={1984},
  publisher={Springer}
}

@book{aubin1999set,
  title={Set-valued Analysis},
  author={Aubin, J.-P. and Frankowska, H. },
  year={1999},
  publisher={Springer}
}

@book{BacciottiRosier2005,
  author    = {Bacciotti, A. and Rosier, L.},
  title     = {{Liapunov} Functions and Stability in Control Theory},
  edition   = {2nd},
  publisher = {Springer},
  address   = {Berlin, Heidelberg},
  year      = {2005},
}

@article{beck2009fast,
  author  = {Beck, A. and Teboulle, M.},
  title   = {A Fast Iterative Shrinkage-Thresholding Algorithm for Linear Inverse Problems},
  journal = {SIAM J. Imaging Sci.},
  volume  = {2},
  number  = {1},
  pages   = {183--202},
  year    = {2009},
  doi     = {10.1137/080716542}
}

@article{beneduci2016positive,
  title={Positive operator valued measures and {F}eller {M}arkov kernels},
  author={Beneduci, R.},
  journal={J. Math. Anal. Appl.},
  volume={442},
  number={1},
  pages={50--71},
  year={2016},
  publisher={Elsevier}
}

@article{Bounkhel2002,
  author  = {Bounkhel, M.},
  title   = {Existence results of nonconvex differential inclusions},
  journal = {Port. Math.},
  volume  = {59},
  number  = {3},
  pages   = {283--310},
  year    = {2002},
}

@article{BurkeFerris1993,
  author  = {Burke, J. V. and Ferris, M. C.},
  title   = {Weak sharp minima in mathematical programming},
  journal = {SIAM J. Control Optim.},
  volume  = {31},
  number  = {5},
  pages   = {1340--1359},
  year    = {1993},
  doi     = {10.1137/0331063}
}

@article{camlibel2007lyapunov,
  title={Lyapunov stability of complementarity and extended systems},
  author={Camlibel, M. K. and Pang, J.-S. and Shen, J.},
  journal={SIAM J. Optim.},
  volume={17},
  number={4},
  pages={1056--1101},
  year={2007},
  publisher={SIAM}
}

@article{caozhangpoor21,
  author  = {X. Cao and J. Zhang and H. V. Poor},
  title   = {Online stochastic optimization with time-varying distributions},
  journal = {IEEE Trans. Automat. Control},
  volume  = {66},
  number  = {4},
  pages   = {1840--1847},
  year    = {2021},
}

@article{ChenLYZ,
  title={Stochastic approximation methods for the two-stage stochastic linear complementarity problem},
  author={Chen, L. and Liu, Y. and Yang, X. and Zhang, J.},
  journal = {SIAM J. Optim.},
  volume={32},
  pages={2129-2155},
  year={2022},
  publisher={SIAM}
}

@article{ChenTKW2017,
  title={ Two-stage stochastic variational inequalities: an {ERM}-solution procedure},
  author={Chen, X. and Pong, T.K. and Wets, R.J-B.},
  journal = {Math. Program.},
  volume={165},
  pages={71-112},
  year={2017},
  publisher={Springer}
}

@article{ChenSS,
  title={Convergence analysis of sample average approximation of two-stage stochastic generalized equations},
  author={Chen, X. and Shapiro, A. and Sun, H.},
  journal = {SIAM J. Optim.},
  volume={29},
  pages={135-161},
  year={2019},
  publisher={SIAM}
}

@article{chen2022dynamic,
  title={Dynamic stochastic variational inequalities and convergence of discrete approximation},
  author={Chen, X. and Shen, J.},
  journal={SIAM J. Optim.},
  volume={32},
  number={4},
  pages={2909--2937},
  year={2022}
}

@article{chen2013,
title={ Convergence of regularized time-stepping methods for differential variational inequalities},
author={X. Chen and Z. Wang},
journal={SIAM J. Optim.},
volume =	 23,
  year =	 2013,
  pages =	 {1647--1671},
}

@article{chen2014,
title={    Differential variational inequality approaches to dynamic games with shared constraints},
author={X. Chen and Z. Wang},
journal={Math. Program.},
volume =	 146,
  year =	 2014,
  pages =	 {379--408},
}

@article{ChenXiang2011,
  title={Implicit solution function of {$P_0$} and {Z} matrix linear complementarity constraints},
  author={Chen, X. and Xiang, S.},
  journal = {Math. Program.},
  volume={128},
  pages={1-18},
  year={2011},
  publisher={Springer}
}

@article{ChenXiang,
  title={  Newton iterations in implicit time-stepping scheme for differential linear complementarity systems},
  author={Chen, X. and Xiang, S.},
  journal = {Math. Program.},
  volume={138},
  pages={579-606},
  year={2013},
  publisher={Springer}
}

@book{CoverThomas2006,
  author    = {T. M. Cover and J. A. Thomas},
  title     = {Elements of Information Theory},
  edition   = {2nd},
  publisher = {Wiley},
  address   = {Hoboken, NJ},
  year      = {2006}
}

@article{cutlerdrusvyatskiyharchaoui23,
  author  = {J. Cutler and D. Drusvyatskiy and Z. Harchaoui},
  title   = {Stochastic optimization under distributional drift},
  journal = {J. Mach. Learn. Res.},
  volume  = {24},
  pages   = {1--56},
  year    = {2023},
}

@book{deimling2011multivalued,
  title={Multivalued Differential Equations},
  author={Deimling, K.},
  volume={1},
  year={2011},
  publisher={Walter de Gruyter}
}

@article{donchev1998stability,
  title={Stability and {E}uler approximation of one-sided Lipschitz differential inclusions},
  author={Donchev, T. and Farkhi, E.},
  journal = {SIAM J. Control Optim.},
  volume={36},
  number={2},
  pages={780--796},
  year={1998},
  publisher={SIAM}
}

@article{dontchev1992difference,
  title={Difference methods for differential inclusions: {A} survey},
  author={Dontchev, A. and Lempio, F.},
  journal={SIAM Rev.},
  volume={34},
  number={2},
  pages={263--294},
  year={1992},
  publisher={SIAM}
}

@article{DontchevFarkhi1989,
  author  = {Dontchev, A. L. and Farkhi, E. M.},
  title   = {Error estimates for discretized differential inclusions},
  journal = {Computing},
  year    = {1989},
  volume  = {41},
  number  = {4},
  pages   = {349--358},
  doi     = {10.1007/BF02241223}
}

@book{dontchev2009implicit,
  title={Implicit Functions and Solution Mappings},
  author={Dontchev, A. L. and Rockafellar, R. T.},
  volume={543},
  year={2009},
  publisher={Springer}
}

@book{DM2018MarkovChains,
  title     = {Markov Chains},
  author    = {Douc, R. and Moulines, E. and Priouret, P. and Soulier, P.},
  year      = {2018},
  publisher = {Springer},
  edition   = {1},
}

@article{fukushima1992equivalent,
  title={Equivalent differentiable optimization problems and descent methods for asymmetric variational inequality problems},
  author={Fukushima, M.},
  journal = {Math. Program.},
  volume={53},
  pages={99--110},
  year={1992},
  publisher={Springer}
}

@article{giuffre2023harnessing,
  title={Harnessing the power of synthetic data in healthcare: innovation, application, and privacy},
  author={Giuffr{\`e}, M. and Shung, D. L.},
  journal={NPJ Digit. Med.},
  volume={6},
  number={1},
  pages={186},
  year={2023},

}

@book{hale2009ordinary,
  title={Ordinary Differential Equations},
  author={Hale, J. K.},
  year={1980},
  publisher={Courier Corporation}
}

@article{halmich2025data,
  title={Data augmentation of time-series data in human movement biomechanics: {A} scoping review},
  author={Halmich, C. and H{\"o}schler, L. and Schranz, C. and Borgelt, C.},
  journal={PloS One},
  volume={20},
  number={7},
  pages={e0327038},
  year={2025},
}

@article{Han-JSPang2024,
title={Continuous Selections of Solutions to Parametric Variational Inequalities},
author={S. Han and J.-S. Pang},
journal={SIAM J. Optim.},
volume =	 34,
  year =	 2024,
  pages =	 {870--892},
}

@article{hossin2015review,
  title={A review on evaluation metrics for data classification evaluations},
  author={Hossin, M. and Sulaiman, M. N.},
  journal={Int. J. Data Min. Knowl. Manag. Process},
  volume={5},
  number={2},
  pages={1},
  year={2015},
}

@article{jordon2022synthetic,
  title={Synthetic Data--what, why and how?},
  author={Jordon, J. and Szpruch, L. and Houssiau, F. and Bottarelli, M. and Cherubin, G. and Maple, C. and Cohen, S. N. and Weller, A.},
  journal={arXiv preprint arXiv:2205.03257},
  year={2022}
}

@article{jorgensen2024disease,
  title={Disease trajectories from healthcare data: methodologies, key results, and future perspectives},
  author={J{\o}rgensen, I. F. and Haue, A. D. and Placido, D. and Hjaltelin, J. X. and Brunak, S.},
  journal={Annu. Rev. Biomed. Data Sci.},
  volume={7},
  number={1},
  pages={251--276},
  year={2024},
}

@book{kallenberg1997foundations,
  title={Foundations of Modern Probability},
  author={Kallenberg, O.},
  volume={2},
  year={1997},
  publisher={Springer}
}

@article{kim2023efficient,
  title={Efficient assessment of real-world dynamics of circadian rhythms in heart rate and body temperature from wearable data},
  author={Kim, D. W. and Mayer, C. and Lee, M. P. and Choi, S. W. and Tewari, M. and Forger, D. B.},
  journal={J. R. Soc. Interface},
  volume={20},
  number={205},
  pages={20230030},
  year={2023},
}

@article{LCH2009,
title={Solving Optimization-constrained Differential Equations with Discontinuity Points,
with Application to Atmosphereic Chemistry},
author={C. Landry and A. Caboussat and E. Hairer},
journal={SIAM J. Sci. Comput.},
 number =	 5,
  volume =	 31,
  year =	 2009,
  pages =	 {3806--3826},
}

@article{liu2025embodied,
  author    = {Liu, H. and Guo, D. and Cangelosi, A.},
  title     = {Embodied Intelligence: A Synergy of Morphology, Action, Perception and Learning},
  journal   = {ACM Comput. Surv.},
  volume    = {57},
  number    = {7},
pages = {186:1--186:36},
  year      = {2025},
}

@article{LuoChen,
  title={Dynamic systems coupled with solutions of stochastic nonsmooth convex optimization},
  author={Luo, J. and Chen, X.},
  journal = {SIAM J. Control Optim.},
  volume={63},
  pages={1686-1708},
  year={2025},
  publisher={SIAM}
}

@article{mccaw2024synthetic,
  title={Synthetic surrogates improve power for genome-wide association studies of partially missing phenotypes in population biobanks},
  author={McCaw, Z. R. and Gao, J. H. and Lin, X. H. and Gronsbell, J.},
  journal={Nat. Genet.},
  volume={56},
  number={7},
  pages={1527--1536},
  year={2024},
}

@book{molchanov2005theory,
  title={Theory of Random Sets},
  author={Molchanov, I.},
  volume={19},
  number={2},
  year={2017},
  publisher={Springer}
}

@article{JSPang2010,
title={Three modeling paradigms in mathematical programming},
author={J.-S. Pang},
journal={Math. Program.},
volume =	 125,
  year =	 2010,
  pages =	 {297--323},
}

@article{PangDVI,
title={Differential Variational Inequalities},
author={J.-S. Pang and D. E. Stewart},
journal={Math. Program.},
volume =	 113,
  year =	 2008,
  pages =	 {345--424},
}

@article{RW,
  title={Stochastic variational inequalities: single-stage to multistage},
  author={Rockafellar, R. T. and Wets,R.J-B.},
  journal = {Math. Program.},
  volume={165},
  pages={331-360},
  year={2017},
  publisher={Springer}
}

@book{rockafellar2009variational,
  title={Variational Analysis},
  author={Rockafellar, R. T. and Wets, R. J.-B.},
  volume={317},
  year={2009},
  publisher={Springer Science \& Business Media}
}

@article{rockwood2007frailty,
  title={Frailty in relation to the accumulation of deficits},
  author={Rockwood, K. and Mitnitski, A.},
  journal={J. Gerontol. A Biol. Sci. Med. Sci.},
  volume={62},
  number={7},
  pages={722--727},
  year={2007},
  publisher={Oxford University Press}
}

@article{rujas2025synthetic,
  author  = {Rujas, M. and {Herranz}, R. M. G. D. M. and Fico, G. and {Merino-Barbancho}, B.},
  title   = {Synthetic data generation in healthcare: {A} scoping review of reviews on domains, motivations, and future applications},
  journal = {Int. J. Med. Inform.},
  volume  = {195},
  pages   = {105763},
  year    = {2025},
  doi     = {10.1016/j.ijmedinf.2024.105763}
}

@article{shapiro2007uniform,
  title={Uniform laws of large numbers for set-valued mappings and subdifferentials of random functions},
  author={Shapiro, A. and Xu, H.},
  journal={J. Math. Anal. Appl.},
  volume={325},
  number={2},
  pages={1390--1399},
  year={2007},
  publisher={Elsevier}
}

@article{shevitz1994lyapunov,
  title={Lyapunov stability theory of nonsmooth systems},
  author={Shevitz, D. and Paden, B.},
  journal={IEEE Trans. Autom. Control},
  volume={39},
  number={9},
  pages={1910--1914},
  year={1994},
  publisher={IEEE}
}

@article{terazono2015continuity,
  title={Continuity of optimal solution functions and their conditions on objective functions},
  author={Terazono, Y. and Matani, A.},
  journal={SIAM J. Optim.},
  volume={25},
  number={4},
  pages={2050--2060},
  year={2015},
  publisher={SIAM}
}

@book{Tsybakov2009,
  author    = {A. B. Tsybakov},
  title     = {Introduction to Nonparametric Estimation},
  publisher = {Springer},
  address   = {New York},
  year      = {2009}
}

@article{yannelis1990upper,
  title={On the upper and lower semicontinuity of the {{A}umann} integral},
  author={Yannelis, N. C.},
  journal={J. Math. Econ.},
  volume={19},
  number={4},
  pages={373--389},
  year={1990},
  publisher={Elsevier}
}

\end{document}